\documentclass[a4paper,11pt]{amsart}
\usepackage[foot]{amsaddr}
\usepackage{geometry}
\usepackage{amsfonts}
\usepackage{mathrsfs}
\usepackage{amsthm}
\usepackage{amssymb}
\usepackage{amsmath}
\usepackage{theoremref}
\usepackage{enumerate}
\usepackage{bbm}
\usepackage{bm}
\usepackage{mathtools}
\usepackage{a4wide}
\usepackage{xcolor}
\usepackage{lipsum}

\makeatletter
\newcommand{\proofpart}[2]{%
  \par
  \addvspace{\medskipamount}%
  \noindent\emph{Step #1: #2}\par\nobreak
  \addvspace{\smallskipamount}%
  \@afterheading
}
\makeatother

\DeclarePairedDelimiter\abs{\lvert}{\rvert}%
\DeclarePairedDelimiter\norm{\lVert}{\rVert}%

\makeatletter
\let\oldabs\abs
\def\abs{\@ifstar{\oldabs}{\oldabs*}}
\let\oldnorm\norm
\def\norm{\@ifstar{\oldnorm}{\oldnorm*}}
\makeatother

\makeatletter
\g@addto@macro\bfseries{\boldmath}
\makeatother

\newcommand{\C}{\mathbb{C}}

\newcommand{\G}{\mathcal{G}}

\newcommand{\T}{\partial\mathbb{D}}

\newcommand{\Ka}{\mathcal{K}}
\newcommand{\conj}[1]{\overline{#1}}
\newcommand{\D}{\mathbb{D}}
\newcommand{\B}{\mathcal{B}}
\newcommand{\Po}{\mathcal{P}}

\newcommand{\cD}{\conj{\mathbb{D}}}

\newcommand{\hil}{\mathcal{H}}

\renewcommand\Re{\operatorname{Re}}

\newtheorem{thm}{Theorem}[section]
\newtheorem{lemma}[thm]{Lemma}

\newtheorem{cor}[thm]{Corollary}
\newtheorem{prop}[thm]{Proposition}

\theoremstyle{definition}

\theoremstyle{definition}
\newtheorem{definition}{Definition}[section]
\newtheorem{remark}[thm]{Remark}

\begin{document}
\title{\textbf{Asymptotic polynomial approximation in the Bloch space}}
\author{Adem Limani} 
\address{Departament de Matem\`atiques, Universitat Autònoma de Barcelona.} 
\email{AdemLimani@mat.uab.cat}

\date{\today}

\maketitle

\begin{abstract}
We investigate asymptotic polynomial approximation for a class of weighted Bloch functions in the unit disc. Our main result is a structural theorem on asymptotic polynomial approximation in the unit disc, in the flavor of the classical Plessner Theorem on asymptotic values of meromorphic functions. This provides the appropriate framework to study metric and geometric properties of sets $E$ on the unit circle for which the following simultaneous approximation phenomenon occurs: there exists analytic polynomials which converge (pointwise or in mean) to zero on $E$ and to a non-zero function in the weighted Bloch norm. We shall offer a characterization completely within the realm of real-analysis, establish a connection to removable sets for analytic Sobolev functions in the complex plane, and provide several necessary conditions in terms of entropy, Hausdorff content and condenser capacity. Furthermore, we demonstrate two principal applications of our developments, which go in different directions. First, we shall deduce a rather subtle consequence in the theme of smooth approximation in de Branges-Rovnyak spaces. Secondly, we answer some questions that were raised almost a decade ago in the theory of Universal Taylor series. 

\end{abstract}

\section{Introduction}
\subsection{Asymptotic polynomial approximation in the unit disc}
Let $\D$ denote the unit disc in the complex plane $\C$ and let $X$ be a Banach space of holomorphic functions in $\D$ which contains the polynomials as a dense subset. In this note, we shall investigate Lebesgue measurable subsets $E \subset \T$ for which the following simultaneous approximation phenomenon arises: \emph{for any $f\in X$ and any bounded function $h$ on $E$, there exists analytic polynomials $(Q_n)_n$ such that $Q_n \to f$ in $X$ and $Q_n \to h$ (pointwise or in mean) on $E$.}
Conversely, we are interested in describing sets $E \subset \T$ giving rise to the following rigidity property: whenever $(Q_n)_n$ are analytic polynomials converging to $0$ in $X$ and converging to some function $f$ in $E$ (pointwise or in mean), then $f=0$ identically. Interchanging the roles between "convergence on $E$" and "convergence in $X$" in the previous statement, one encounters the perhaps more familiar concept of the Khinchin-Ostrowski property. This common theme of notions is what we shall refer to as \emph{asymptotic polynomial approximation}, and they typically fit within the general scheme of uncertainty principles in harmonic analysis, see \cite{havinbook}. These problems have previously been considered in different contexts and a principal source where the setting of holomorphic growth spaces was treated, is due to S. Khrushchev in \cite{khrushchev1978problem}. Our main intention is to carry out similar investigations in the setting of the classical Bloch space $\B$ consisting of holomorphic functions $f$ satisfying
\[
\norm{f}_{\B}:= \abs{f(0)} + \sup_{z\in \D} (1-|z|)\abs{f'(z)}.
\]
Since our problem involves convergence of analytic polynomials, we are, in fact, confined within the little Bloch space $\B_0$, which is the separable subspace of $\B$ consisting of functions $f$ satisfying 
\[
\lim_{|z| \to 1-} (1-|z|)f'(z) =0.
\]
\subsection{Notions and general framework}

As indicated in the introduction, the precise modes of convergence on subsets $E \subseteq \T$ in the different notions of asymptotic polynomial approximation will not crucial. However, in order to maintain a broad point of view, we shall find it more convenient to consider weak-star convergence in the space of essentially bounded functions on $E$, instead of uniform convergence on $E$. By passing to appropriate convex combinations of polynomials, one gets convergence in any $p$-th mean and pointwise $dm$-a.e in $E$. The upshot is that one can work with general Lebesgue measurable subsets $E$ of $\T$, as the space of essentially bounded functions on $E$ is well-defined, whereas the uniform closure of polynomials on non (locally) compact subsets $E$ are rather pathological. 

In what follows, all subsets of $\T$ of our considerations will be Lebesgue measurable, hence we shall for the sake of abbreviation just refer to them as sets. As sets of zero Lebesgue measure will be irrelevant in our considerations, we say that a set $K$ is \emph{almost contained} in a set $E$ if $K\setminus E$ has Lebesgue measure zero. Let $1\leq p\leq \infty$ and $E\subseteq \T$ be a set, and denote by $L^p(E)$ the usual Lebesgue space on $E$ wrt the Lebesgue measure $dm$ on $\T$. It will be convenient to regard $L^p(E)$ as the closed subspace of $L^p(\T)$, whose elements vanish $dm$-a.e off $E$. Let $\text{Hol}(\D)$ be the space of holomorphic of functions in $\D$ equipped with the topology of uniform convergence on compact subsets of $\D$ and let $X \hookrightarrow \text{Hol}(\D)$ denote a Banach space which contains the polynomials as a dense subset. For a set $E \subset \T$, we denote by $\Po_E(X)$ be the closure of the diagonal 
\[
\mathcal{Q}:= \left\{(Q,Q): Q \, \, \text{analytic polynomials} \right\} 
\]
in the space $X \oplus L^\infty(E)$, equipped with the topology inherited from the norm in $X$ and the weak-star topology in $L^\infty(E)$. Here we recall that that a sequence of polynomials $(Q_n)_n$ converges weak-star to an essentially bounded function $f\in L^\infty(E)$ if 
\[
\lim_n \int_E Q_n h dm = \int_E fh dm, \qquad h\in L^1(E).
\]
Our principal object of study is the outcome of the space $\Po_E(X)$. A set $E$ satisfies the \emph{SA-property} (simultaneous approximation) wrt $X$ if $\Po_E(X) = X \oplus L^\infty(E)$. We shall sometimes also refer to such sets as \emph{SA-sets} for $X$. A subset $E \subseteq \T$ is said to be a \emph{set of rigidity} for $X$, if the projection $\Pi: X \oplus L^\infty(E) \to X$ restricted to the subspace $\Po_E(X)$ is injective. In other words, $E$ is a set of rigidity for $X$ if whenever $(Q_n)_n$ are analytic polynomials such that $Q_n \to 0$ in $X$ and $Q_n \to f$ weak-star in $L^\infty(E)$, then $f=0$. The notion of rigidity originates from the conception that $\Po_E(X)$ should be regarded a genuine space of analytic functions, since each element $(f, f^*) \in \Po_E(X)$ is uniquely determined by its holomorphic function $f$ on $\D$. Whenever $E$ is a set of rigidity for $X$, we shall say that $\Po_E(X)$ is \emph{irreducible}. The concepts of SA-sets and sets of rigidity should be thought of as opposite phenomenons, as the former property requires $\Po_E(X)$ to be "maximally large", while the latter requires $\Po_E(X)$ to be as small as possible. We shall also need a slightly weaker notion of the classical Khinchin-Ostrowksi property, which originates from the Khinchin-Ostrowksi theorem, see \cite{havinbook}. A set $E\subseteq \T$ is said to satisfy the \emph{weak KO-property} wrt $X$ if whenever $(Q_n)_n$ are analytic polynomials with the property that $Q_n \to f$ in $X$ for some bounded function $f$ and $Q_n \to 0$ weak-star in $L^\infty(E)$, then $f=0$. In other words, the weak KO-property refers to that every bounded element in $(f, f^*) \in \Po_E(X)$ is uniquely determined by its "boundary function" $f^*$. The principle aim of our work revolves around understanding when these different phenomenons occur and the relationship between them.

\subsection{Background}
Given a number $1\leq p< \infty$, we recall the definition of the classical Bergman spaces $L^p_a(\D)$, which consist of holomorphic functions in $\D$ which are $L^p$-integrable wrt to the Lebesgue area measure $dA$ on $\D$, that is
\[
\norm{f}^p_{L^p(dA)} := \int_{\D} \abs{f(z)}^p dA(z) < \infty.
\]
In the magnificent PhD-thesis of Sergei Khrushchev, announced in 1978, the problem of simultaneous approximation was thoroughly investigated in the context holomorphic functions $\D$ that satisfy certain radial growth, which include the classical Bergman space, see \cite{khrushchev1978problem}. In his work, a deep connection between simultaneous approximation and to a certain (one-sided) problem on removable singularities for Cauchy transforms, was establish, which ultimately allowed him to obtain a complete geometric descriptions on sets satisfying the SA-property and the Khinchin-Ostrowski property in that setting. In order to efficiently illustrate his result, we paraphrase an accessible summary of his work, and refer the reader to \cite{khrushchev1978problem} for further details.

\begin{thm}[S. Khrushchev, 1978]\thlabel{THM:HRUSH}
Let $1\leq p< \infty$ and $E \subseteq \T$ be a Lebesgue measurable set. Then the following statements are all equivalent: 
\begin{enumerate}
    \item[(i)] $E$ does not satisfy the SA-property wrt $L^p_a(\D)$.
    \item[(ii)] $E$ satisfies the Khinchin-Ostrowski property wrt $L^p_a(\D)$.
    \item[(iii)]
    There exists a compact set $K$ of positive Lebesgue measure, which is almost contained in $E$, such that $K$ has finite Beurling-Carleson entropy:
\[
\sum_k \abs{I_k} \log \frac{1}{\abs{I_k}} < \infty
\]
where $\{I_n\}_n$ denotes the connected components of $\T \setminus K$.
\item[(iv)] There exists a non-trivial integrable function $\phi$ on $\T$ that vanishes off $E$, such that its Cauchy transform
\[
C(\phi)(z) := \int_E \frac{\zeta \phi(\zeta) dm(\zeta)}{\zeta -z}, \qquad z\in \D
\]
extends to a non-constant function in $C^\infty(\T)$.
\end{enumerate}
\end{thm}
\noindent
Here $C^\infty(\T)$ denotes the set of $C^\infty$-smooth functions on $\T$. Sets of finite Beurling-Carleson entropy which have zero Lebesgue measure are well-known to be boundary zero sets of holomorphic functions in $\D$ which extend smoothly up to $\T$, for instance, see \cite{carlesonuniqueness} and \cite{taylor1970ideals}. Khrushchev's Theorem does not only give a characterisation of sets satisfying the Khinchin-Ostrowski principle in terms of the Beurling-Carleson entropy, but also establishes a deep connection to the problem of supports of functions with smooth Cauchy transform, which can typically be nowhere dense in $\T$. In our language, S. Khrushchev proved that sets of finite Beurling-Carleson entropy are sets of rigidity for the Bergman spaces. In fact, his results were actually phrased within the general setting of weighted growth spaces $\G_W$ of holomorphic functions $g$ in $\D$ satisfying
\begin{equation}\label{DEF:GW}
\norm{g}_{\G_W} := \sup_{z\in \D} W(1-|z|) \abs{g(z)} < \infty.
\end{equation}
Here one should think of $W$ as a non-decreasing weight function on the unit-interval $[0,1)$ with $W(0)=0$, which tends to zero slower than an exponential. Khrushchev's general result is then analogues to \thref{THM:HRUSH}, where $(iii)$ is now expressed in terms of an adapted $W$-entropy condition (see \eqref{wset} below) and $(iv)$ involves a statement on where the Cauchy transform belongs to an appropriate Cauchy-dual space of $\G_W$. We refer the reader to \cite{khrushchev1978problem} for further details.

The problem of simultaneous approximation in the setting of growth spaces $\G_W$ is also related to the Lusin-Privalov uniqueness problem of radial limits, which essentially asserts that there exists in any growth space $\G_W$, a non-trivial holomorphic function $f$ with zero radial limit at $dm$-a.e on $\T$. In fact, S. Khrushchev proved that this result can be utilized to give a simple proof of the existence of non-trivial SA-sets for $\G_W$, while B. Korenblum and C. Beneteau gave a quantitative proof using a sharp Jensen-type inequality, to show that such behavior can only occur on sets which do not contain any compact subsets of positive measure having finite $W$-entropy, see \cite{beneteau2001jensen}.

On a related matter, we shall now mention the recent work of B. Malman on Thomson decompositions in the context of $\Po^t(d\mu)$-spaces, see \cite{malman2023thomson}. Given $1\leq t < \infty$ and a set $E \subset \T$ of positive Lebesgue measure, we are interested in the outcome by taking the closure of analytic polynomials in the Lebesgue spaces $L^t(d\mu)$, denoted by $\Po^t(d\mu)$, among measures $d\mu$ of the form
\[
d\mu(z) = 1_{\D}(z)W(1-|z|)dA(z) + 1_E(z) dm(z).
\]
A deep result by J. Thomson in \cite{thomson1991approximation}, implies that for any set $E$, there exists a partition of $E$ into subsets $A, S$, which is unique up to sets of Lebesgue measure zero, such that the following decomposition holds
\begin{equation}\label{EQ:Thomdec}
\Po^t(d\mu) = \Po^t(1_{\D}WdA+ 1_A dm) \oplus L^t (1_S dm),
\end{equation}
interpreted as isometric isomorphism of Banach spaces, where the space $\Po^t(1_{\D} WdA+ 1_A dm)$ is \emph{irreducible}, in the sense that it contains no non-trivial characteristic function. Thomson's decomposition in \eqref{EQ:Thomdec} allows one to single out the "bad" $L^t(1_Sdm)$-summand from a $\Po^t(d\mu)$-spaces, in such a way that one is always left with an irreducible space. In the profound work of A. Aleman, S. Richter and C. Sundberg in \cite{aleman2009nontangential}, it was proved that irreducible $\Po^t(d\mu)$-spaces are genuine spaces of holomorphic functions that share many familiar traits with the classical Hardy spaces. More generally, $\Po^2 (d\mu)$-spaces play a fundamental role in the theory of subnormal operators, as they serve as functional models therein, see \cite{conway1991theory}. Now the principal contribution in the work of B. Malman was a rather precise geometric description of the subsets $A, S$, relative to the set $E$, which appear in the Thomson decomposition. In a similar vain as the work of Khrushchev, his description involves the notion of $W$-entropy. Roughly speaking, the subset $A$ of $E$ is the "maximal" set formed by subsets of $E$ having finite $W$-entropy, while any subset of $S$ has infinite $W$-entropy. We refer the reader to \cite{malman2023thomson} for further details.

From a function theory perspective, the Bloch space is viewed as limiting case of the classical Bergman spaces, hence it is natural to consider the corresponding problems in this setting. In fact, it was mentioned by N. Makarov that it would be desirable to have a corresponding description of sets satisfying the SA-property wrt to the Bloch space, and furthermore, it was indicated that this problem could also be related to the problem of describing supports of smooth measures, see \cite{makarov1989class}. To the authors best knowledge, even the question on existence of sets having positive Lebesgue measure and satisfying the SA-property wrt $\B_0$, has previously not been confirmed. An instant obstacle is the observation that the connection to the Lusin-Privalov uniqueness problems of radial limits no longer holds. Indeed, the Lehto-Virtanen maximal principle for normal functions (for instance, see \cite{pommerenke2013boundary}) and Privalov's Theorem, ensures that any Bloch function which has boundary value $0$ along some curve(s) on a set of positive Lebesgue measure, must vanishes identically. It also turns out that existence SA-sets for the Bloch space would also have interesting applications to certain problems on universal Taylor series, see \cite{beise2016generic}. The author has also been informed in private communication with J. Bruna, that sets satisfying the Khinchin-Ostrowski property in the Bloch space setting had previously received some attention in the early 90's, but no official progress was ever declared.

\section{Main results}
We summarize our main results in subsections featuring their principal contents. In a similar vain as the work of S. Khrushchev, several of our findings will favorably be phrased within a broader context of weighted Bloch spaces.

\subsection{A structural theorem and existence of SA-sets}
We declare a \emph{majorant} $w$ to be a continuous non-decreasing function on the unit-interval $[0,1]$ with the property that there exists a number $0<\alpha <1$, such that $w(t)/t^{\alpha}$ increases up to $\infty$ as $t\to 0+$. Given a majorant $w$, we define the space $\B_0(w)$ consisting of analytic functions $f$ in $\D$ satisfying 
\[
\lim_{\abs{z}\to 1-} \frac{1-|z|}{w(1-|z|)}f'(z) = 0.
\]
In other words, the weighted Bloch-type space $\B_0(w)$ is actually the closure of the analytic polynomials taken in the norm 
\[
\norm{f}_{\B(w)} := \abs{f(0)} + \sup_{z\in \D} \frac{1-|z|}{w(1-|z|)}\abs{f'(z)}.
\]
In the case when the majorant $w$ is a constant, one retrieves the classical little Bloch space $\B_0$, which plays a crucial role in geometric function theory, for instance, see \cite{garnett2005harmonic}. If $w(t)=t^\alpha$ then a classical result of G.H. Hardy and J. Littlewood asserts that $\B(w)$ is the space of holomorphic functions in $\D$ which are H\"older continuous on $\cD$, see \cite{duren1970theory}. It turns out that non-trivial SA-sets for a weighted Bloch space $\B_0(w)$ can only arise when the associated majorant $w$ tends to zero slowly enough. Our first result takes the following form.

\begin{thm}\thlabel{THM:SAsets} Let $w$ be a majorant and $\B_0(w)$ denote the corresponding Bloch-type space. Then there exists a set of positive Lebesgue measure which satisfies the SA-property wrt $\B_0(w)$ if and only if  
\begin{equation}\label{DiniDiv}
\int_0^1 \frac{w^2(t)}{t}dt = \infty.
\end{equation}
If condition \eqref{DiniDiv} does not hold, then any set of positive Lebesgue measure is a set of rigidity for $\B_0(w)$.
\end{thm}
\noindent

Recall that the existence of SA-sets in $\B_0(w)$ is no longer related the Lusin-Privalov uniqueness problem, hence our proof of \thref{THM:SAsets} provides an alternative route via composition operators of inner functions with asymptotically small hyperbolic derivatives. Furthermore, we shall illustrate in section \ref{SEC:7} and in section \ref{SEC:8}, that the existence of SA-sets for $\B_0(w)$ has some interesting consequences to the theory of de Branges-Rovnyak spaces and to the theory of Universal Taylor series, respectively. Moving forward, we now phrase our next result, which provides a general structural theorem for asymptotic polynomial approximation in the weighted Bloch spaces, which is admittedly similar in spirit to the classical Plessner Theorem on the asymptotic behaviour of meromorphic functions. 
\begin{thm}\thlabel{THM:Bstruc} Let $w$ be a majorant and $E \subset \T$ be a set of positive Lebesgue measure. Then there exists a unique (modulo sets of Lebesgue measure zero) partition into subsets $A,S$ of $E$, depending only on $E$ and $w$, for which the following statements hold:
\begin{enumerate}
    \item[(i)] $A$ is a set of rigidity for $\B_0(w)$. 
    \item[(ii)] $S$ satisfies the SA-property wrt $\B_0(w)$.
\end{enumerate}
Moreover, the following isomorphic identification holds:
\begin{equation}\label{EQ:CHAUM}
    \Po_{E} \B_0(w) \cong \Po_A \B_0(w) \oplus L^\infty(S).
\end{equation}
Modulo sets of Lebesgue measure zero, the measurable sets $A,S$ appearing in the partition of $E$ are conformally invariant.
\end{thm}

Here it is very important to remark that the sets appearing in the statement of \thref{THM:Bstruc} are relative to the given initial set $E$. Roughly speaking, this means that a set $E$ satisfies the SA-property wrt $\B_0(w)$ if and only if $E$ completely lacks "rigidity content" wrt $\B_0(w)$. The isomorphic decomposition in \eqref{EQ:CHAUM} may be regarded as a version of Chaumat's lemma (see Lemma VII.1.7 in \cite{conway1991theory}) and should be contrasted with a version of Thomson's decomposition in \eqref{EQ:Thomdec}. Note that \thref{THM:Bstruc} provides no geometric conditions on the corresponding subsets $A, S$ of $E$, nor does it exclude the possibility that either $A$ or $S$ can be of zero Lebesgue measure, hence does not alone assure the existence of non-trivial SA-sets for $\B_0(w)$. The last statement asserts that for any automorphism $\varphi$ on $\D$, $\varphi(A)$ is set of rigidity for $\B_0(w)$ and $\varphi(S)$ is an SA-set for $\B_0(w)$.

\subsection{A function theoretical description}
We now turn our attention to the description of the various notions of asymptotic polynomial approximation in $\B_0(w)$. To this end, we define the special atomic space $\mathscr{B}_w(\T)$, which consists of distributions $g$ on $\T$, which can be expressed as 
\[
g(\zeta) = \sum_n c_n b_n(\zeta), \qquad \zeta \in \T
\]
where $\{c_n\}_n$ is an absolutely summable complex-valued sequence and $\{b_n\}_n$ are special atoms supported on corresponding arcs $\{I_n\}_n$ on $\T$ and defined by $b_0 \equiv 1$ and
\[
b_n (\zeta) = \frac{1}{|I_n|w(|I_n|)} \left( 1_{I_n^+}(\zeta)-1_{I_n^-}(\zeta) \right), \qquad \zeta \in \T, \qquad n\neq 0,
\]
where $I_n^-, I_n^+$ denote the left and right halves of $I_n$, respectively. The special atomic space $\mathscr{B}_w(\T)$, equipped with the norm 
\[
\norm{g}_{\mathscr{B}_w} = \inf \left\{ \sum_n \abs{c_n}: g= \sum_n c_n b_n \right\}
\]
where the infimum is taken over all such representations of $g$, becomes a Banach space. These spaces were introduced by S. Bloom and S. de Souza in \cite{bloom1989atomic}, and have several applications in harmonic analysis, as for $w=1$, they essentially describe elements with absolutely summable Haar expansions. For instance. see \cite{de1986several} and \cite{de1989fourier}. 
In our next result, we shall offer a function theoretical description of SA-sets for $\B_0(w)$ in terms of the space $\mathscr{B}_w(\T)$. To this end, we recall the definition of the Hilbert transform $H$ of a function $\phi \in L^1(\T)$ as
\[
H(\phi)(e^{i\theta}) := \lim_{\varepsilon \to 0+} \int_{\left\{e^{i\theta}:\abs{t-\theta}>\varepsilon \right\}}  \cot \left( \frac{t-\theta}{2}\right) \phi(e^{it}) \frac{dt}{2\pi},
\]
which exists for $dm$-almost every $e^{i \theta}$ on $\T$, see \cite{garnett}. Our function theoretical description of sets satisfying the SA-property wrt $\B_0(w)$ takes the following form.

\begin{thm}\thlabel{THM:CHAR1} Let $w$ be a majorant satisfying \eqref{DiniDiv} and $E \subseteq \T$ be a set of positive Lebesgue measure. Then following statements are all equivalent.
\begin{enumerate}
    \item[(i)] $E$ does not satisfy the SA-property wrt $\B_0(w)$.
    \item[(ii)] $E$ satisfies the weak KO-property wrt $\B_0(w)$.
    \item[(iii)] $E$ almost contains a non-trivial subset $A$ of positive Lebesgue measure, such that $A$ is a set of rigidity for $\B_0(w)$.
    \item[(iv)] There exists a non-trivial pair of real-valued functions $u, v \in L^1(E)$ such that 
\[
u + H(v) \in \mathscr{B}_w(\T) \setminus \{0\},
\]
where $H$ denotes the Hilbert transform on $\T$.

\end{enumerate}

Furthermore, $E$ is a set of rigidity for $\B_0(w)$ if and only if $E$ is a carrier set of at least one of the functions $u,v$ appearing in $(iv)$.

\end{thm}
\noindent
The first three properties encapsulates the precise relationships between the different notions of asymptotic polynomial approximation, while $(iv)$ shows that they can all be captured within the realm of real-analysis, and phrased as a function theory problem. Note that $(iv)$ implies that if a Borel set $E$ supports a non-trivial $L^1(\T)$-function in $\mathscr{B}_w(\T)$, then $E$ cannot be an SA-set wrt $\B_0(w)$. Indeed, this observation is rather simple to state and follows by taking the corresponding companion $v=0$ in $(iv)$, but the condition phrased in \thref{THM:CHAR1} is actually a much weaker assumption and does not require that one puts the entire burden on $u$ alone. Even if the set $E$ is so "rough" that it does not support an $L^1(\T)$-function in $\mathscr{B}_w(\T)$, the Hilbert transform is a convolution operator which "spreads out" the mass, hence the companion $v$ can only aid in making the sum $u+H(v)$ land in $\mathscr{B}_w(\T)$. The precise connection between $\mathscr{B}_w(\T)$ and the corresponding Bloch-type space $\B_0(w)$ will be clarified in Section \ref{SEC:5}. We remark that in the classical setting of the Bloch space $\B_0$ (with $w=1$), we shall below see that the corresponding special atomic space $\mathscr{B}_1(\T)$ corresponds to an atomic decomposition of $L^1(\T)$-functions belonging to a certain Besov space $B^1(\T)$ (see \eqref{DEF:B^1} for definition below), hence we may also rephrase $(iv)$ to that setting. 


 
\subsection{Removable sets and capacitary conditions}
Our next objective is to illustrate a connection between simultaneous approximation and removable sets for analytic Sobolev spaces in the plane, which will allow us to complement the function theoretical characterization in \thref{THM:CHAR1}, by giving metric descriptions of SA-sets. To this end, we shall solely restrict our attention to the classical setting of the Bloch space $\B_0$. Let $1\leq p < \infty$ and $\Omega \subseteq \C$ be a domain and define the analytic Sobolev space $W^p_a(\Omega)$, consisting of holomorphic functions $f$ in $\Omega$ such that 
\[
\int_{\Omega} \abs{f'(z)}^p dA(z) < \infty.
\]
Here $dA$ denotes the Lebesgue area measure on $\C$, restricted to $\Omega$. A compact set $K\subset \C$ is declared to be \emph{removable} for $W^1_a$ if the space $W^p_a(\C \setminus K) $ only consists of constants. The study of removable sets for analytic Sobolev spaces has a vast history, originating from the profound work of L. Ahlfors and A. Beurling in \cite{ahlfors1950conformal}, where the classical Dirichlet space $W^2_a$ was studied. Their results contained various descriptions of removable sets in the Dirichlet setting $W^{2}_a$, involving notations of (condenser) capacities and extremal distances, using Hilbert space methods and precise reproducing formulas. One their result commonly highlighted in the literature, asserts that a compact set $K \subset \T$ is removable for $W^2_a$ if and only if
\begin{equation}\label{EQ:ABNCAP}
\textbf{Cap}(I\setminus K) = \textbf{Cap}(I)
\end{equation}
for any arc $I \subseteq \T$, where $\textbf{Cap}(\cdot)$ denotes either the logarithmic capacity or the $1/2$-Bessel capacity. See also \cite{makarov1989class}. Later, L. Hedberg extended their work to the separable setting $1<p<\infty$, providing descriptions of similar kind, but using duality methods and Calder\'on-Zygmund theory, see \cite{hedberg1974removable}. Very little is known about removable sets for $W^1_a$, as this case seems resistant to methods previously potent in the setting of $p>1$. Although, a detailed investigation of this question is beyond the scope of our work, we shall below demonstrate a connection between removable sets for $W^1_a$ to sets satisfying the SA-property in $\B_0$. To put our result into a broader context, we mention the work of S. Khrushchev in \cite{khrushchev1978problem}, where he proved that compact subsets $K \subset \T$ which are not removable for $W^2_a$ must necessarily contain subset of finite Beurling-Carleson entropy, hence cannot satisfy the SA-property wrt any Bergman space. Later, N. Makarov gave a more streamlined proof of this fact, by establishing a deep connection between notions of entropy and Hausdorff contents (see \cite{makarov1989class} and subsection 2.4 below). In order to give an analogue in the Bloch space setting, we shall introduce a certain Besov space. Let $t>0$ and consider the second translated difference of a function $f$ on $\T$, defined as
\[
\Delta_2(f,t)(\zeta):= f(\zeta e^{it})+f(\zeta e^{-it})-2f(\zeta), \qquad \zeta \in \T.
\]
We denote by $B^1(\T)$ the Besov space, which consists of functions $g\in L^1(\T)$ equipped with the semi-norm 
\begin{equation}\label{DEF:B^1}
\norm{g}_{B^1} = \int_0^1 \int_{\T} \abs{\Delta_2(G,t)(\zeta)} dm(\zeta) \frac{dt}{t^2} < \infty,
\end{equation}
where $G(e^{it})= \int_0^t g(e^{is}) ds$ denotes a primitive of $g$. In the broader context, primitives of functions in $B^1(\T)$ are elements in the diagonal Besov space $B^1_{1,1}(\T):= B^{\alpha}_{p,p}(\T)$ (a typical notation in the literature) defined on the unit circle $\T$, with Lebesgue parameter(s) $p=1$ and smoothness parameter $\alpha=1$, making it an end-point case wrt both these parameters. It turns out that $B^1(\T) \cong \mathscr{B}_1(\T)$, in the sense of being isomorphic Banach spaces with equivalent norms, were the later space may be viewed as an atomic decomposition of $B^1(\T)$, see Section \ref{SEC:3}. Consider the Zygmund class $\Lambda_1(\T)$, consisting of continuous functions $f$ on $\T$ satisfying 
\[
\norm{f}_{\Lambda_1} := \sup_{\zeta \in \T, t>0} \frac{\abs{\Delta_2(f,t)(\zeta)}}{t} < \infty.
\]
Taking the closure of trigonometric polynomials in the semi-norm of $\Lambda_1(\T)$, one retrieves the "little"-o Zygmund class $\lambda_1(\T)$, which consists of continuous function $f$ on $\T$ satisfying
\[
\lim_{\delta \to 0+} \sup_{\zeta \in \T, 0<t< \delta} \frac{\abs{\Delta_2(f,t)(\zeta)}}{t} =0.
\]
It turns out that one can identify the dual space of $\mathscr{B}_1(\T)$ (hence also $B^1(\T)$) with $\Lambda_1(\T)$, considered in the densely defined pairing:
\begin{equation}\label{EQ:ZYGBWDual}
h \mapsto \int_{\T} h(\zeta) f'(\zeta) dm(\zeta), \qquad h\in \mathscr{B}_1(\T), \qquad f \in C^\infty(\T).
\end{equation}
We refer the reader to \cite{bloom1989atomic} for further details. Given a compact set $K$, we denote by $C^\infty_K(\T)$ the set of functions $\phi \in C^\infty(\T)$, whose derivative $\phi'$ is compactly supported in $\T \setminus K$. Given an arc $I \subseteq \T$, we define the $\Lambda_1$-condenser capacity of $I$ with respect to a compact set $K$ by 
\begin{equation}\label{EQ:LwCondCap}
\textbf{Cap}_{\Lambda_1}(I / K) := \inf_{\phi  \in C^\infty_K(\T) } 
\, \norm{\phi}_{\Lambda_1},
\end{equation}
where the infimum is taken over all $\phi \in C_K^\infty(\T)$ with $\phi =0 $ on one end-point of $I$ and $\phi=1$ on the other end-point. The ordinary $\Lambda_1$-condenser capacity of an arc $I$ is denoted by $\textbf{Cap}_{\Lambda_1}(I)$, and defined as the infimum over $\phi \in C^\infty(\T)$ with the above prescribed end-point values on $I$. This definition should be viewed as an $\Lambda_1$-analogue of a condenser-type capacity introduced by L. Hedberg in the context of Sobolev spaces, where the notation $I/K$ arises from the conception that the set $K$ is "quotient out" in the subcollection $C_K^\infty(\T)$ with the above prescribed properties, see \cite{hedberg1974removable}. Our next result gives several characterizations of removable sets for $W^1_a$, which are confined in the unit-circle, and may be viewed as an end-point analogue to the aforementioned results of Ahlfors, Beurling and Hedberg.


\begin{thm}\thlabel{THM:CHARREM} Let $K \subset \T$ be a compact set of positive Lebesgue measure. Then the following statements are equivalent:
\begin{enumerate}
    \item[(i)] $K$ is removable for the analytic Sobolev space $W^1_a$.
    \item[(ii)] There exists no non-trivial function $g \in B^1(\T)$ (or in $\mathscr{B}_1(\T))$, which is supported in $K$.
    \item[(iii)] $C^\infty_K(\T)$ is norm-dense in the little Zygmund class $\lambda_1(\T)$.
\end{enumerate}
Furthermore, if $K$ is negligible for any $\Lambda_1$-condenser capacity in the sense that condition
\begin{equation}\label{REMB1}
     \textbf{Cap}_{\Lambda_1}(I/K) = \textbf{Cap}_{\Lambda_1}(I)
\end{equation}
holds for any arc $I\subseteq \T$, then $K$ is removable for $W^1_a$.

\end{thm}
Note that condition $(ii)$ asserts that the only impediment for a compact subset of $\T$ to be removable for $W^1_a$ is that it supports a non-trivial element in $B^1(\T)$ or $\mathscr{B}_1(\T)$. This should be contrasted with statement $(iv)$ in \thref{THM:CHAR1} on SA-sets for $\B_0$, which is seemingly weaker. Condition $(iii)$ gives a description in terms of an approximation problem in the little-o Zygmund spaces $\lambda_1(\T)$, which may as well be reformulated in the setting of weak-star density in the Zygmund space $\Lambda_1(\T)$. In particular, it implies that no compact subset of $K$ satisfying $(iii)$ can every support a non-trivial function in $\Lambda_1(\T)$. Note that it also reveals that the problem of removable sets for $W^1_a$ is intrinsically a local problem within the realm of real-analysis, and a similar description as $(iii)$ also appears in the context of removable sets for $W^p_a$ in \cite{hedberg1974removable}. We conjecture that condition \eqref{REMB1} may give yet another characterization of removable sets for $W^1_a$, but we have unfortunately not been able to verify this claim. One can show that the ordinary $\Lambda_1$-condenser capacity of an arc $I$ is attained for a trigonometric polynomial $p_I$ of degree $1$, but due to lack of convexity in our setting, we cannot ensure that $C^\infty_K(\T)$-functions approximate any $p_I$ in weak-star of $\Lambda_1$. Since there is no natural Bessel-potential space which corresponds to the Zygmund class $\Lambda_1$, we can unfortunately not offer any closer analogue to the Ahlfors-Beurling condition in \eqref{EQ:ABNCAP}, besides \eqref{REMB1}. However, we invite the reader to compare \label{REMB1} to the condition involving Hausdorff contents in \thref{COR:Hwinv} below. Now comparing part $(iv)$ in \thref{THM:CHAR1} with $(ii)$ of \thref{THM:CHARREM}, we obtain an intimate connection between sets satisfying the SA-property wrt $\B_0$ and removable sets for $W^1_a$, similar in spirit to Khrushchev's result.
\begin{cor}\thlabel{THM:REMSA} 
If a compact subset $K \subset \T$ of positive Lebesgue measure satisfies the SA-property wrt $\B_0$, then it is removable for the analytic Sobolev space $W^1_a$. 
\end{cor}
\noindent
Note that the existence of non-trivial compact sets which are removable for $W^1_a$ is ensured by \thref{THM:SAsets}. Part $(iv)$ of \thref{THM:CHAR1} indicates that SA-sets for $\B_0$ are likely to be more complicated than removable set for $W^1_a$, as viewed from $(ii)$ in \thref{THM:CHARREM}. In fact, we shall see that SA-sets for $\B_0$ can be rephrased as a one-sided problem of removable singularities for Cauchy transforms belonging to $W^1_a(\D)$, while removable sets for $W^1_a$ boils down to a two-sided problem for Cauchy transforms (in and outside the unit disc). Similar observations in the context of growth spaces were obtained by Khrushchev, see \cite{khrushchev1978problem}. Our next observation provides a more practical necessary condition for a set to be removable for $W^1_a$, hence also an SA-set for $\B_0$. It essentially asserts that sets which are too evenly distributed in $\T$ cannot be removable for $W^1_a$, and are in fact sets of rigidity for $\B_0$.

\begin{prop}\thlabel{PROP:NecDR} Let $E \subset \T$ be a Lebesgue measurable set satisfying the condition
\begin{equation}\label{EQ:Symcond}
\int_0^1 \abs{(E+t) \triangle E} \, \frac{dt}{t} < \infty,
\end{equation}
then $E$ almost contains a compact subset which is not removable for $W^1_a$, and $E$ is a set of rigidity for $\B_0$. 

\end{prop}
\noindent
Here $E+t :=\{ e^{i(s+t)}: e^{is} \in E \}$ is the translate of $E$ and $\triangle$ denotes the symmetric difference between sets. In particular, the proposition implies that if $E$ satisfies the SA-property wrt $\B_0$, then it cannot almost contain any subset $K$ which satisfies condition \eqref{EQ:Symcond}. For instance, we shall see that \thref{PROP:NecDR} gives a simple proof that sets of finite Beurling-Carleson entropy of positive Lebesgue measure are not removable for $W^1_a$.




\subsection{Entropy conditions and Hausdorff contents}
Given majorant $w$, a compact set $K \subseteq \T$ of positive Lebesgue measure is said to be a \emph{set of finite $w$-entropy} if the following condition holds:
\begin{equation}\label{wset}
\sum_n |I_n| \log w(|I_n|) > -\infty
\end{equation}
where $\{I_n\}_n$ are the connected components of $\T \setminus K$. As previously mentioned, S. Khrushchev proved that sets of finite $w$-entropy can be used to characterize SA-sets in the context of growth spaces $\G_w$, see \cite{khrushchev1978problem}. Moreover, sets of finite $w$-entropy having zero Lebesgue measure, are actually the boundary zero sets of holomorphic functions in $\D$ which extend continuously up to $\cD$, having modulus of continuity not exceeding the majorant $w$ there. See \cite{shirokov1982zero}. More recently, sets of finite $w$-entropy of zero Lebesgue measure zero were shown to play a decisive role in the theory of shift invariant subspaces in growth spaces, \cite{limani2023m_z}. Let $\zeta \in \T$ and denote by
\[
\Gamma(\zeta):= \left\{z\in \D: \abs{z-\zeta} \leq \sigma (1-|z|) \right\}
\]
the Stolz domain in $\D$ associated to $\zeta$ of some fixed aperture $\sigma>1$. Given a compact set $K \subset \T$, we set $\Gamma_K$ to be the Privalov-Stolz domain in $\D$ associated to $K$, defined by
\begin{equation}\label{EQ:PrivStolz}
    \Gamma_K := \cup_{\zeta \in E} \Gamma(\zeta),
\end{equation}
where each $\Gamma(\zeta)$ are assumed to have the same fixed aperture $\sigma >1$. Of course, each $\sigma >1$ gives rise to different Privalov-Stolz domains, but as they their specific shape will not be important for us, we leave it unspecified. Our next result gives a geometric sufficient condition for compact sets $K \subset \T$ to be a sets of rigidity for $\B_0$ and more.

\begin{thm}\thlabel{THM:suffwG} Let $\gamma>0$ be a number such that $w^\gamma$ is a majorant satisfying the Dini condition
\begin{equation}\label{WDiniCond}
\int_0^1 \frac{w^{\gamma}(t)}{t} dt < \infty.
\end{equation}
Then any set of finite $w$-entropy is a set of rigidity for $\B_0$. Furthermore, for any $(f , f^*) \in \Po_K \B_0(w)$, the holomorphic function $f$ on $\D$ has non-tangential limit equal to $f^*$ $dm$-a.e on $K$, and the following embedding holds
\[
\Po_K \B_0(w) \hookrightarrow H^\infty(\Gamma_K),
\]
for any Stolz-Privalov domain $\Gamma_K$, where $H^\infty(\Gamma_K)$ denotes the Hardy space of bounded holomorphic functions on $\Gamma_K$.
\end{thm}
In particular, if a Lebesgue measurable set $E$ almost contains a set $K$ of finite $w$-entropy, then $E$ satisfies the weak KO-property wrt $\B_0$, hence cannot satisfy the SA-property wrt $\B_0$. Example of eligible majorants $w$ are, for instance, $t^{\min{1,\alpha}}, \log^{-\alpha}(e/t)$ with $\alpha>0$. The second statement in \thref{THM:suffwG} provides an intimate relationship between $f$ and its boundary function $f^*$, which shows that irreducible spaces $\Po_K \B_0$ behave like the classical Hardy spaces, which underpins a stark contrast between sets of rigidity and sets satisfying the SA-property. This is the analogue to the theory of irreducible $\Po^t(d\mu)$-spaces, mentioned in the introduction. Our next observation essentially asserts that the Dini-condition on the majorant $w$ cannot be improved much further. 

\begin{thm} \thlabel{THM:SHARPENT}
Let $w$ be a majorant satisfying 
\[
\int_{0}^1 \frac{w^2(t)}{t} dt = \infty,
\]
but such that condition \eqref{WDiniCond} holds for some $\gamma >2$. Then there exists compact sets $K$ which satisfies the SA-property wrt $\B_0$ and almost contains no subset of finite $w$-entropy.
    
\end{thm}
\noindent
The clarify our main point, consider the majorant $w(t) := \log^{-c}(e/t)$ for some $0<c\leq 1/2$, and note that \thref{THM:suffwG} ensures that every subset of finite $w$-entropy is a set of rigidity for $\B_0$, while \thref{THM:SHARPENT} shows that there are sets almost containing no subset of finite $w$-entropy and satisfy the SA-property wrt $\B_0$. It is deemed unlikely that any entropy condition alone can characterize SA-sets for $\B_0$. On a related matter, it was recently established that the notation of entropy alone fails do describe cyclic vectors in the Bloch space, see \cite{limani2023mzinvariant}. We shall rephrase the entropy condition above in terms of a certain condition involving Hausdorff contents. To this end, $\beta$ be a non-decreasing continuous function on the unit-interval $[0,1]$ with $\beta(0)=0$. For a Lebesgue measurable set $E \subseteq \T$, we define the $\beta$-Hausdorff content of $E$ as the quantity
\[
\hil_\beta (E) := \inf \left\{ \sum_k \beta (|I_k|): \,
\{I_k\}_k \, \text{covering of open arcs of E} \right\}.
\]
Again, the $\beta$-Hausdorff content satisfies the usual properties of monotoncity, sub-additivity and one verifies in a straightforward manner that 
\[
\hil_{\beta}(I) = \beta(|I|)
\]
for any arc $I \subset \T$. A deep result of N. Makarov (see Theorem in \S 2, \cite{makarov1989class}) reveals a close relationship between the notion of entropy and $\beta$-Hausdorff content, asserting that a set $E \subset \T$ almost contains no subset $K$ of finite $w$-entropy if and only if 
\[
\hil_\beta(I \setminus E) = \hil_\beta(I)
\]
for any arc $I\subseteq \T$, where $\beta(t) = - t \log w(t)$. Sets $E$ satisfying the above condition are sometimes said to be \emph{$w$-sparse}. Using Makarov's description of $w$-sparse sets, we obtain the following feature of SA-sets in $\B_0$, as an immediate corollary of \thref{THM:suffwG}.  

\begin{cor} \thlabel{COR:Hwinv}
If $E$ satisfies the SA-property wrt $\B_0$, then $E$ is $w$-sparse for any majorant $w$ satisfying a Dini-condition in \eqref{WDiniCond}. That is, 
\[
\hil_\beta(I) = \hil_\beta(I \setminus E)
\]
for any arc $I\subseteq \T$, where $\beta(t) = - t \log w(t)$. Moreover, the Dini-condition cannot be improved in the sense of \thref{THM:SHARPENT}.
\end{cor}
Roughly speaking, \thref{COR:Hwinv} asserts that sets satisfying the SA-property wrt $\B_0$ must be negligible for Hausdorff contents wrt a range of measure-functions $\beta$ which tend to zero sufficiently fast.

\subsection{Organization and notations}

The manuscript is organized as follows. In Section \ref{SEC:3} we shall broaden our framework and rephrase the problems of asymptotic polynomial approximation by means of duality, with the principal intent of establishing \thref{THM:Bstruc} and parts of \thref{THM:CHAR1}. In fact, they will be deduced as consequences of our more general results in \thref{THM:AbsThom} and \thref{THM:KOprop}, respectively. Section \ref{SEC:4} is devoted to settling \thref{THM:SAsets} on the existence of sets of positive Lebesgue measure which satisfy the SA-property wrt $\B_0(w)$. In section \ref{SEC:5}, we prove our function theoretical characterizations of SA-sets in \thref{THM:CHAR1} and our description of removable sets for $W^1_a$ in \thref{THM:CHARREM}. Section \ref{SEC:6} contains the proof of \thref{THM:suffwG}, which involves the necessary condition for SA-sets in terms of entropy and Hausdorff content, and indicate its sharpness by proving \thref{THM:SHARPENT}. The principal applications of our developments are contained in \ref{SEC:7} and \ref{SEC:8}, where the former section is concerned with approximation in de Branges-Rovnyak spaces, while the later section revolves around Menshov universality of Taylor series.

Throughout these notes, we shall shall frequently use the notation $A \lesssim B$ to declare that two positive numbers $A,B$ satisfy the relation $A\leq cB$ for some positive constant $c>0$.  In case $A \lesssim B$ and $B \lesssim A$, we write $A\asymp B$.

\subsection*{Acknowledgements} I am in debt to Joaquim Bruna-Flores for bringing this problem to my attention and for encouraging me to pursue it with caution. Moreover, I thank Artur Nicolau and Oleg Ivrii for fruitful discussions during the preparation of this manuscript.

\section{A general model for asymptotic polynomial approximation} \label{SEC:3}
In this section, we shall conveniently phrase our results within a more general framework. In fact, we insist that maintaining a broader point of view will improve clarity in expositions.

\subsection{Holomorphic Banach spaces and Cauchy duals} Let $\text{Hol}(\D)$ denote the space of holomorphic of functions in $\D$ equipped with the topology of uniform convergence on compact subsets of $\D$. We shall restrict our attention to Banach spaces $X \hookrightarrow \text{Hol}(\D)$ of holomorphic functions in $\D$ satisfying the following properties:
\begin{enumerate}
    \item[(i)] The set of analytic polynomials is dense in $X$.
    \item[(ii)] The shift operator $M_zf(z) = zf(z)$ acts continuously on $X$.
    \item[(iii)] For every $\ell \in X^*$ the following asymptotic holds:
    \[
    \limsup_{n\to \infty} \, \abs{\ell(z^n)}^{1/n} \leq 1,
    \]
    where $X^*$ denotes the Banach space dual of $X$.
\end{enumerate}
Condition $(i)$ ensures that $X$ is a separable Banach space and condition $(ii)$ implies that $X$ is $M_z$-invariant. Meanwhile, condition $(iii)$ allows us to substitute the abstract Banach space dual $X^*$ with continuous linear functionals arising from a space of holomorphic functions in $\D$. Specifically, condition $(i)$ and $(iii)$ imply that any $\ell \in X^*$ gives rise to a uniquely defined holomorphic function $g_\ell$ on $\D$, given by the formula
\[
g_\ell(z) = \sum_{n\geq 0} \conj{\ell(\zeta^n)}z^n \qquad z\in \D,
\]
and we may define the corresponding space of holomorphic functions by $X'$, equipped with the inherited norm 
\[
\norm{g_\ell}_{X'} := \norm{\ell}_{X^*}.
\]
The formed Banach space $X'$ is called the \emph{Cauchy dual} of $X$, where the notion stems from the observation that any bounded linear functional $\ell \in X^*$ acting on analytic polynomials $p$ can be expressed as
\[
\ell(p) = \lim_{r \to 1-} \int_{\T} p(r\zeta) \conj{g_\ell(r\zeta)} dm(\zeta).
\]
In other words, every bounded linear functional $\ell \in X^*$ may be regarded as the corresponding holomorphic function $g_\ell \in X'$, considered in the above Cauchy pairing. We shall refer to Banach spaces $X$ satisfying the conditions $(i)-(iii)$ as \emph{holomorphic Banach spaces}. There are numerous examples of holomorphic Banach spaces, such as the classical Bergman spaces, but the principal examples of our consideration will be the weighted Bloch space $\B_0(w)$, as we shall clarify in the next subsection.


\subsection{Weighted Bloch spaces and their duals}
It is straightforward to verify that the weighted Bloch spaces $\B_0(w)$ satisfy the assumptions $(i)$ and $(ii)$, and one easily checks that the $n$-th moment satisfies
\[
\limsup_{n\to \infty} \norm{z^n}^{1/n}_{\B(w)} \leq 1,
\]
which implies that $(iii)$ also holds. Since these spaces play a crucial role in our work, we shall need a more explicit description of their Cauchy duals. Here we use the relation $X \cong Y$ between that two Banach spaces $X, Y$ to designate that they are isomorphic with equivalent norms. 

\begin{lemma}\thlabel{LEM:BlochwDual} Let $w$ be a majorant and let $W^1_a(w)$ denote the Banach space of holomorphic functions $f$ in $\D$ equipped with the norm 
\[
\norm{f}_{W^1_a(w)} := \abs{f(0)} + \int_{\D} \abs{f'(z)} w(1-|z|)dA(z).
\]
Then the following relationships on Cauchy duality hold: $\B_0(w)' \cong W^1_a(w)$ and $W^1_a(w)' \cong \B(w)$.
 
\end{lemma}

\begin{proof}
    Note that a straightforward application of the Littlewood-Paley formula gives 
    \[
    \int_{\T}f(r\zeta) \conj{g(r\zeta)} dm(\zeta) = f(0) \conj{g(0)} + r^2 \int_{\D} f'(rz) \conj{g'(rz)} \log \frac{1}{|z|} dA(z), \qquad 0<r<1
    \]
    whenever $f, g$ are holomorphic functions in $\D$. It follows easily from this formula and some standard estimates that every $g \in \B(w)$ induces a bounded linear functional on $W^1_a(w)$ and that every $f \in W^1_a(w)$ induces a bounded linear functional on $\B_0(w)$. To show the converse, we primarily note that a straightforward argument involving dilatation shows that the analytic polynomials are dense in $W^1_a(w)$. The next step is to verify that if $g$ is holomorphic in $\D$ for which 
    \[
    \ell_g(f) := \lim_{r\to 1-} \int_{\T} f(r\zeta) \conj{g(r\zeta)} dm(\zeta), \qquad f\in W^1_a(w)
    \]
    induces a bounded linear functional on $W^1_a(w)$, then $g$ must necessarily belong to $\B(w)$. To this end, we apply $\ell_g$ to the rational functions $f_{\lambda}(\zeta) = (1-\conj{\lambda}\zeta)^2$ with $\lambda \in \D$ fixed and utilize the Cauchy integral formula to get
    \[
    \abs{g'(\lambda)} = \lim_{r\to 1-} \abs{\int_{\T}\frac{g(r\zeta)}{(1-\conj{r\zeta}\lambda)^2} dm(\zeta) } = \abs{\ell_g(f_\lambda)} \leq C \norm{f_{\lambda}}_{W^1_a(w)}.
    \]
    Here $C>0$ is a constant, which arises from the assumption $\ell_g \in W^1_a(w)'$ and is independent of $f$. It remains to compute the $W^1_a(w)$-norm of $f_{\lambda}$, which by polar coordinates takes the form 
    \[
    \norm{f_{\lambda}}_{W^1_a(w)} = 1 + \int_{0}^1  \int_{\T} \frac{dm(\zeta)}{\abs{\zeta - r\lambda}^3}  w(1-r) rdr.
    \]
    Now the integral on $\T$ is asymptotically equal to $(1-r\abs{\lambda})^{-2}$ (cf. Chap. I, \cite{hedenmalmbergmanspaces}). Using this in conjunction with a change of variable, and the assumption that $w$ is a majorant, one gets
    \[
    \int_{0}^1  \int_{\T} \frac{dm(\zeta)}{\abs{\zeta - r\lambda}^3}  w(1-r) rdr \lesssim \int_{0}^1 \frac{w(1-r)}{(1-r|\lambda|)^2} dr \leq \int_{1-\abs{\lambda}}^1 \frac{w(t)}{t^2} dt \lesssim \frac{w(1-|\lambda|)}{1-|\lambda|}.
    \]
    We therefore conclude that $g \in \B(w)$, hence $W^1_a(w) ' \cong \B(w)$. Now it is a general result that $\B_0(w)'$ is the unique pre-dual of $\B(w)$ (cf. Corollary 1.3 in \cite{perfekt2017m}), hence using the previously established relation, we conclude that $\B_0(w)' \cong W^1_a(w)$.
\end{proof}

The next results reveals the intimate link between the space special atomic spaces $\mathscr{B}_w(\T)$, previously defined in Section 2.2, to the weighted Sobolev spaces $W^1_a(w)$. 

\begin{thm}\thlabel{THM:BWchar}[Theorem 9.1 in \cite{bloom1989atomic}] Let $w$ be a majorant. Then for any function $F\in W^1_a(w)$ the corresponding limit
\[
\lim_{r\to 1-} \Re F(r\zeta) =: f(\zeta) \qquad \zeta \in \T
\]
exists $dm$-a.e and defines an element in $\mathscr{B}_w(\T)$. Conversely, for any $f \in \mathscr{B}_w(\T)$ the corresponding Cauchy projection 
\[
C(f)(z) := \int_{\T} \frac{f(\zeta)}{1- \conj{\zeta}z} dm(\zeta), \qquad z\in \D
\]
belongs to $W^1_a(w)$.
\end{thm}
It now readily follows from \thref{LEM:BlochwDual} and \thref{THM:BWchar} that the weighted Bloch spaces satisfy property $(iv)$, hence they are holomorphic Banach spaces. We now clarify the role of the Besov space $B^1(\T)$ in our work. Recall that $B^1(\T)$ consists of $g\in L^1(\T)$ satisfying the condition
\begin{equation*}
\norm{g}_{B^1} = \int_0^1 \int_{\T} \abs{G(\zeta e^{it}) +G(\zeta e^{-it})  - 2G(\zeta)} dm(\zeta) \frac{dt}{t^2} < \infty,
\end{equation*}
where $G(e^{it})= \int_0^t g(e^{is}) ds$ is a primitive of $g$. The following results is essentially gives a Sobolev trace theorem for $W^1_a(\D)$.

\begin{thm}[See Theorem A in \cite{holland1988growth}] \thlabel{PROP:W11B1} 
The Cauchy transform is a continuous surjective linear operator from the Besov space $B^1$ onto the analytic Sobolev $W^1_a(\D)$. That is, we have
\[
C (B^1 ) = W^1_a(\D).
\]
Furthermore, $W^1_a(\D)$ consists of holomorphic functions $g$ in $\D$ such that $g^*(\zeta) :=\lim_{r\to 1-} g(r\zeta)$ defines an element in $B^1$.
\end{thm}

An immediate consequence of \thref{PROP:W11B1} and \thref{THM:BWchar} is that the spaces $B^1(\T)$ and $\mathscr{B}_1(\T)$ are isomorphic as Banach spaces, with equivalent norms, that is
\[
B^1(\T) \cong \mathscr{B}_1(\T).
\]
In fact, one may actually view $B^1(\T)$ as the space of absolutely summable Haar-type expansions of elements in $\T$.

%
%
\subsection{Asymptotic polynomial approximation in holomorphic Banach spaces} 
Here we phrase the relevant notions of asymptotic polynomial approximation in the general setting of holomorphic Banach spaces $X$. Given a Lebesgue measurable set $E \subseteq \T$ of positive Lebesgue measure, we form the space $X \oplus L^\infty(E)$, naturally equipped with the inheriting norm-topology from $X$ and weak-star topology from $L^\infty(E)$. The dual space of $X \oplus L^\infty(E)$ can be identified with $X' \oplus L^1(E)$ in the pairing 
\[
\lim_{r \to 1-} \int_{\T} F(r\zeta) \conj {G(r\zeta)} dm(\zeta) + \int_E f(\zeta) \conj{g(\zeta)} dm(\zeta),
\]
with $(F,f) \in X \oplus L^\infty(E)$ and $(G,g) \in X' \oplus L^1(E)$. The closure of the $(Q,Q)$ in $X \oplus L^\infty(E)$, where $Q$ ranges through the set of analytic polynomials, is denoted by $\Po_E(X)$. Below we highlight and comment on the different notions of asymptotic polynomial approximation.
\begin{definition}\thlabel{DEF:SAX}
    A Lebesgue measurable subset $E \subseteq \T$ is said to satisfy the simultaneous approximation property (SA-property) wrt $X$ if 
    \[
    \Po_E(X) = X \oplus L^\infty(E).
    \]
\end{definition}
\noindent
A simple and useful observation (c.f. Theorem 1.3,  \cite{khrushchev1978problem}) is that a subset $E$ satisfies the SA-property wrt $X$ if and only if the tuple $(0,1)$ belongs to $\Po_E(X)$. In order for a set $E \subseteq \T$ to satisfy the SA-property wrt $X$, the set of analytic polynomials must necessarily be weak-star dense in the space $L^\infty(E)$, which by a straightforward argument involving the F. and M. Riesz Theorem implies that they cannot have full Lebesgue measure on $\T$. For the sake of brevity, we sometimes simply say that $E$ is an SA-set for $X$ and that it is a non-trivial if $E$ has positive Lebesgue measure. We now recall the concept which should be thought of as the opposite counterpart to simultaneous approximation.  

\begin{definition}\thlabel{DEF:IRRX} A Lebesgue measurable set $E \subset \T$ is declared to be a set of rigidity for $X$ if the projection $\Pi : X \oplus L^\infty(E) \to X$, restricted to the subspace $\Po_E(X)$ is injective.
 \end{definition} 
 In other words, $E$ is a set of rigidity if whenever $(Q_n)_n$ is a sequence of analytic polynomials which converge to $0$ in $X$ and converge to $f$ uniformly on $E$, then $f=0$. In that case, the corresponding space $\Po_E(X)$ is said to be irreducible. Note that any set $E\subseteq \T$ of full Lebesgue measure in $\T$ is always a set of rigidity for any holomorphic Banach space $X$. Indeed, any convergent sequence of analytic polynomials in $L^\infty(E)$ in fact converges weak-star in $H^\infty(\D)$ to a holomorphic function $f$ on $\D$, which together with the assumption $X \hookrightarrow \text{Hol}(\D)$ forces $E$ to be a set of rigidity for $X$. 

\subsection{Removable singularities for Cauchy transforms}

Now shall now recapture the notions of asymptotic polynomial approximation by means of duality, which turns out to be related to a certain one-sided problem of removable singularities for Cauchy transforms. The following dual reformulation for the SA-property is crucial in the work of S. Khrushchev in \cite{khrushchev1978problem}. As it will also play an important role for our further developments, we phrase it below and refer to it as Khrushchev's lemma.

\begin{thm}[Khrushchev's Lemma] \thlabel{SAduality} Let $X$ be a holomorphic Banach space and $E \subset \T$ be a Lebesgue measurable set of positive Lebesgue measure. Then $E$ satisfies the SA-property with respect to $X$ if and only if there exists no $h \in L^1(E)$ such that its Cauchy transform
\[
C(h)(z) := \int_E \frac{ h(\zeta) }{1-\conj{\zeta}z}dm(\zeta), \qquad z\in \D
\]
is a non-trivial element in $X'$.

\end{thm}
\noindent 
The proof involves a neat duality argument, similar to Theorem 1.2 in \cite{khrushchev1978problem}, where it is proved for compact sets $E$. The difference there is that simultaneous approximation is defined by requiring uniform convergence on $E$, hence in the space $C(E)$ of continuous functions on $E$, which leads the author to retrieve a Borel measure $h$ supported on $E$ with $C(h) \in X'$. However, since we only require weak-star convergence in $L^\infty(E)$, we will retrieve an integral functions supported in $E$.

Let $C(L^1)$ denotes the Banach space of holomorphic functions $f$ in $\D$, expressible as the Cauchy transform of a elements in $L^1(\T)$, equipped with the inherited norm
\[
\norm{f}_{C(L^1)} := \inf \left\{ \norm{h}_{L^1}: f=C(h) \right\}.
\]
See \cite{cauchytransform} for further details on Cauchy transforms. Our next Theorem provides the corresponding dual reformulation for sets of rigidity, analogously to Khrushchev's lemma on SA-sets.

\begin{thm}\thlabel{THM:Irred} Let $X$ be a holomorphic Banach space and $A\subseteq \T$ be a Lebesgue measurable set with $0<\abs{A}<1$. Then $A$ is a set of rigidity for $X$, if and only if the linear manifold
\[
   \mathcal{L}_A(X') := \left\{ h\in L^1(A): h \neq 0 \, \text{dm-a.e on A}, \, \, \, C(h) \in X' \right\}
\]
is non-empty. Furthermore, the linear manifold 
    \[
  \mathcal{K}_A(X') := \left\{C(h) \in X': h \in L^1(A) \right\}
    \]
forms a dense subset of $C(L^1)$, whenever $A$ is a set of rigidity for $X$.
\end{thm}
In order to deduce \thref{THM:Irred}, we shall need a certain description of $M_z '$-invariant subspaces in $C(L^1)$, where 
\[
M_z 'f(z) = \frac{f(z)-f(0)}{z} \qquad z\in \D.
\]
The result in question essentially follows from the work of A. B. Aleksandrov in \cite{aleksandrovinv}, see also Ch. 5.6 in \cite{cima2000backward} for a detailed exposition. But since it is rephrased quite differently there, we shall for the sake of clarity include a brief sketch of its proof.

\begin{lemma}\thlabel{LEM:Mz*K} Let $\mathcal{M}$ be a non-trivial closed $M_z '$-invariant subspace of $C(L^1)$. Then there exists a unique inner function $\Theta$, such that $\mathcal{M}$ is the kernel of the Toeplitz operator 
\[
T_{\conj{\Theta}}(h)(z) := \int_{\T} \frac{\conj{\Theta(\zeta)}h(\zeta)}{1 - \conj{\zeta} z}dm(\zeta), \qquad z\in \D,
\]  
where $h \in L^1(\T)$. In other words, 
\[
\mathcal{M} = \left\{ f= C(h): T_{\conj{\Theta}}(h) =0, \, \, h\in L^1(\T) \right\}.
\]
\end{lemma}

\begin{proof}[Proof-sketch:] Recall that the pre-dual of $C(L^1)$, considered in the Cauchy-dual pairing, can be identified with $H^\infty$. Note that the pre-annihilator $\mathcal{M}_\perp$ of $\mathcal{M}$, regarded as a weak-star closed subspace of $H^\infty$, is invariant under the shift operator $M_z$. An application of Beurling's Theorem in $H^\infty$ (c.f Theorem 7.5 in \cite{garnett}) implies that there exists a unique inner function $\Theta$, such that $\mathcal{M}_\perp = \Theta H^\infty$. It remains to verify that $\mathcal{M}$ is in the kernel of $T_{\conj{\Theta}}$. However, this claim follows from the observation that the formal transpose map of the multiplication with $\Theta$ on $H^\infty$ equals $T_{\conj{\Theta}}$ on $C(L^1)$.

\end{proof}

\begin{proof}[Proof of \thref{THM:Irred}]
\proofpart{1}{Dual reformulation of sets of rigidity:}

Note that if the set $\mathcal{L}_A(X')$ is empty, then Khrushchev's lemma implies that $A$ satisfies the SA-property wrt $X$, and hence $\Po_A(X)$ cannot be irreducible. Now suppose that the set $\mathcal{L}_A(X')$ is non-empty and we shall in fact argue that it must be dense in $L^1(A)$. Note that since $X$ was assumed to be $M_z$-invariant, a straightforward computation of its formal transpose in $X'$ reveals that $M_z ': X' \to X'$. Furthermore, a routine computation involving Cauchy kernels shows that
\[
M_z ' C(h)(z) = C (\conj{\zeta}h)(z), \qquad z\in \D
\]
whenever $h\in L^1(\T)$, hence we conclude that $\mathcal{L}_A(X')$ is invariant under the operation $h \mapsto \conj{\zeta}h(\zeta)$. Now if the non-empty set $\mathcal{L}_A(X')$ is not dense in $L^1(A)$, then there exists an element $f\in L^\infty(A)$ which annihilates it, thus
\[
\int_{A} f(\zeta) \zeta^n \conj{h(\zeta)} dm(\zeta) = 0 \qquad n\geq 0
\]
for any $h \in \mathcal{L}_A(X')$. According to the F. and M. Riesz Theorem, $f \conj{h}1_{A}$ must agree with the boundary values of an $H^1$-function, but since $A$ does not have full Lebesgue measure on $\T$, we conclude that $f\conj{h}=0$ on $A$. Since $h\neq 0$ $dm$-a.e on $A$, we are forced to conclude that $f=0$ identically as an element in $L^\infty(A)$, which is absurd and  $\mathcal{L}_A(X')$ is therefore dense in $L^1(A)$. Now let $(Q_n)_n$ be a sequence of analytic polynomials with the property that $Q_n \to 0$ in $X$ and $Q_n \to f$ weak-star in $L^\infty(A)$. Then, for any $h \in L^1(A)$ with the property that $C(h) \in X'$, we have 
\[
\int_A f \conj{h} dm = \lim_n \int_{A} Q_n \conj{h} dm = \lim_n \lim_{r\to 1-} \int_{\T} Q_n(r\zeta) \conj{C(h)(r\zeta)} dm=0.
\]
Now since $\mathcal{L}_A(X')$ was dense in $L^1(A)$, we conclude that $f=0$ on $A$, hence $A$ is a set of rigidity for $X$.

\proofpart{2}{Density of $\Ka_A(X')$:} We primarily note that since $X$ is $M_z$-invariant, an simple computation of the transpose of $M_z$ implies that the Cauchy dual $X'$ is $M'_z$-invariant. This means that $\Ka_A(X')$ regarded as a subset of $C(L^1)$ is $M_z '$-invariant. Now suppose that $\Ka_A(X')$ is not dense in $C(L^1)$, then according to \thref{LEM:Mz*K} there exists a non-trivial inner function $\Theta$ such that 
\[
\textbf{Clos}\left(\Ka_A(X') \right)_{C(L^1)} = \left\{ f= C(h): T_{\conj{\Theta}}(h) =0, \, \, h\in L^1(\T) \right\}
\]
This implies that for any $h \in L^1(A)$ with $C(h) \in X'$ satisfies 
\[
\sum_{n\geq 0} z^n \int_{\T} \conj{\Theta(\zeta)} \zeta^{-n} h(\zeta) dm(\zeta) = T_{\conj{\Theta}}(h)(z)  = 0  \qquad z\in \D.
\]
By the F. and M. Riesz Theorem, $\Theta\conj{h}$ agrees with the boundary values of an $H^1$-function, but since $h$ vanishes off the set $A$ (a set of positive Lebesgue measure) and $\abs{\Theta} =1$ $dm$-a.e on $\T$, we conclude that $h \equiv 0$ for any $h \in L^1(A)$ with $C(h) \in X'$. But this implies that $\Ka_A(X') =\{0\}$, which according to Khrushchev's Lemma implies that $A$ satisfies the SA-property wrt $X$, contradicting the initial assumption that $\Po_A(X)$ was irreducible. We therefore conclude that $\Ka_A(X') = C(L^1)$ and the proof is complete.

\end{proof}

We highlight the dichotomy regarding the linear manifold $\mathcal{L}_A(X')$ that arose in the proof of \thref{THM:Irred}. The assumption $\abs{A}<1$ and together with the invariance of $\mathcal{L}_A(X')$ under the action of the multiplication operator $h \mapsto \conj{\zeta}h(\zeta)$ forces the linear manifold $\mathcal{L}_A(X')$ to be either empty or dense in $L^1(A)$.

\subsection{A structural theorem on asymptotic polynomial approximation}
 
The main result in this section is the proof of \thref{THM:Bstruc}, which here is established in the broader context of holomorphic Banach spaces. 

\begin{thm} \thlabel{THM:AbsThom}
Let $X$ be a holomorphic Banach space and $E$ be Lebesgue measurable subset of $\T$. Then there exists a unique (modulo sets of Lebesgue measure zero) partition of $E$ into subsets $A,S$ satisfying the following properties:
\begin{enumerate}
    \item[(i)] $A$ is a set of rigidity for $X$.
    \item[(ii)] $S$ satisfies the SA-property wrt $X$. 
\end{enumerate}
Moreover,the following isometric isomorphic identification 
\begin{equation}\label{absThomdec}
    \Po_E(X) \cong \Po_A(X) \oplus L^\infty(S).
\end{equation}
\end{thm}
We remark that the decomposition in \eqref{absThomdec} may be viewed as an abstract version of Chaumat's lemma (for instance, see \cite{conway1991theory}) in the broader setting of holomorphic Banach spaces. Again, we emphasize that the decomposition appearing in the statement is relative to the initial set $E$, and, a priori, the corresponding Lebesgue measurable subsets $A, S$ of $E$ are allowed to be of zero Lebesgue measure. For instance, if $X$ is contained in the Nevanlinna class $\mathcal{N}$ then the classical Khinchin-Ostrowski Theorem ensures that every set of positive Lebesgue measure is a set of rigidity for $X$. As previously mentioned, sets of full measure in $\T$ are sets of rigidity for holomorphic Banach spaces $X$.

\begin{proof}[Proof of \thref{THM:AbsThom}] Since the statement is phrased modulo sets of Lebesgue measure zero, and sets of full Lebesgue measure are always sets of rigidity for $X$, we may assume that $0<|E|<1$. 
\proofpart{1}{Constructing sets of rigidity:}

Consider the quantity 
\begin{equation}\label{maxirr}
    \iota(X,E) := \sup \left\{ |A|: A \subseteq E, \, \, \, \Po_{A}(X) \, \, \text{is irreducible}  \right\}.
\end{equation}
We shall for a moment assume that the above set is non-empty. First we note that it is straightforward to verify that if $A_1, A_2$ are both sets of rigidity for $X$, then so is their union $A_1 \cup A_2$. We may therefore pick a chain of subsets $A_N \subseteq A_{N+1}$ of $E$ with the property that $\abs{A_N}$ increase up to $\iota(X,E)$. Setting $A := \cup_N A_N$
we see that $\abs{A} = \iota(X,E)$ and it remains to verify that $\Po_A(X)$ is irreducible. To this end, let $\{Q_n\}_n$ be analytic polynomials which converge to $(0,f)$ in $X \oplus L^\infty(A)$. Now since each $\Po_{A_N}(X)$ is irreducible it follows that $f = 0$ on each $A_N$. But since $\{A_N\}_N$ is an exhaustion of $A$ we get that $1_{A_N} f \to 1_A f$ pointwise on $\T$, thus $f=0$ on $A$ and $\Po_A(X)$ is therefore irreducible.

\proofpart{2}{The subset satisfying the SA-property:}
Set $S := E \setminus A$ in case $A$ is non-trivial and $S=E$ in case $A$ is the empty set, 
and it remains to verify that $S$ satisfies the SA-property wrt $X$. Arguing by duality using \thref{SAduality}, suppose the exists an element $h \in L^1(S)$ such that $C(h)$ is a non-trivial element in $X'$. Now if $S_0$ denotes a carrier set for $h$, then it follows from \thref{THM:Irred} that $\Po_{S_0}(X)$ is irreducible. But then $\Po_{A \cup S_0}(X)$ is also irreducible, which contradictions the maximality of $A$, since $\abs{A} + \abs{S_0} = \abs{A \cup S_0} > \iota(X,E)$. We therefore conclude that $S$ satisfies the SA-property wrt $X$.

\proofpart{3}{Verifying the Chaumat-type decomposition:} Note since $\Po_A(X)$ was proved to be irreducible, we may $\Po_E(X)$ regarded as a closed subspace of $\Po_A(X) \oplus L^\infty(S)$, hence it remains to prove the reverse containment. Let $\ell \in \Po_E(X)^\perp$ be an element in the annihilator of $\Po_E(X) \subseteq X \oplus L^\infty(E)$, regarded as a subspace of $(X \oplus L^\infty(E) )' \cong X' \oplus L^1(E)$. Then $\ell = (f,h) \in X' \oplus L^1(E)$ and we seek the annihilator for which the carrier set of $h$ has maximal Lebesgue measure. 
Note that elements $(f,h) \in \Po_E(X)^\perp$ satisfy the property that 
\[
f_n = \lim_{r \to 1-} \int_{\T} f(r\zeta) \zeta^{-n}dm(\zeta) = - \int_{\T}h(\zeta) \zeta^{-n} dm(\zeta) = - \widehat{h}(n), \qquad n=0,1,2, \dots,
\]
where $f_n$ are the Taylor coefficients of $f$ and $\widehat{h}(n)$ are the Fourier coefficients of $h$. This implies that $-f = C(h)$ with $h\in L^1(E)$ and is a non-trivial element in $X'$, whenever $f$ is. Now since $A\subseteq E$ was the subset for which $\abs{A}$ was maximal and $\Po_A(X)$ irreducible, it follows by \thref{THM:Irred} that there exists an element $h\in L^1(A)$ with $h\neq 0$ $dm$-a.e on $A$ and $C(h) \in X' \setminus \{0\}$, and no carrier set of an element $h \in L^1(E)$ with $(f,h) \in \Po_E(X)^{\perp}$ can exceed $A$. This implies that any $\Po_E(X)^{\perp} \subseteq \Po_A(X)^{\perp} \oplus \{0\}$ (actually equality), where the later expression is regarded as a subspace of $X' \oplus L^1(A) \oplus L^1(S)$. Taking pre-annihilators (c.f. Ch. 4.7 in \cite{rudin1973functional}), we arrive at $\Po_E(X) \supseteq \Po_A(X) \oplus L^\infty(S)$. This completes the proof. 
\end{proof}

The proof of \thref{THM:Bstruc} now readily follows, except for the claim on conformal invariance of the $SA$-sets and sets of rigidity. In order to establish the last part of \thref{THM:Bstruc}, we introduce a lemma which highlights some simple and general properties. 

\begin{lemma}\thlabel{LEM:SA-RIG} Let $X$ be a holomorphic Banach space. Then the following properties hold.

\begin{description}
    \item[\textit{Restriction principle}] If a Lebesgue measurable set $E\subset \T$ satisfies the SA-property wrt $X$, then so does any Lebesgue measurable subset $S \subseteq E$.    
    \item[\textit{Monotonicity principle}] If $\{E_n\}_n$ are sets of rigidity for $X$, then their union $E= \cup_n E_n$ is also a set of rigidity for $X$.
\end{description}
    
\end{lemma}
The restriction principle of SA-sets readily follows by definition, while the monotonicity principle for sets of rigidity follows from the proof of \thref{THM:AbsThom}.

\begin{proof}[Proof of \thref{THM:Bstruc}]
   As an immediate consequence of \thref{THM:AbsThom}, it only remains to establish the claims on conformal invariance of the $SA$-sets and the sets of rigidity. To this end, let $E$ be a Lebesgue measurable set $E\subseteq \T$ which is the carrier set of an element $h \in L^1(\T)$ with $C(h) \in W^1_a(w) \neq \{0\}$. We primarily note that the Cauchy transform of the translate $h_t(\zeta) := h(e^{-it}\zeta)$ is carried on the translated set $E+t = \{e^{it} \zeta: \zeta \in E \}$ and again belongs to $W^1_a(w)$. According to \thref{THM:Irred}, we conclude that sets for rigidity are invariant under rotations. Next, we fix a number $\lambda \in \D$ and consider the automorphism on $\D$, defined by
    \[
    \phi_\lambda (z) = \frac{\lambda-z}{1-\conj{\lambda}z}, \qquad z\in \D
    \]
    and recall that $\phi_\lambda \circ \phi_\lambda (z)= z$. Using a change a variable and carrying out some basic algebraic manipulations, we can show that 
    \[
    C(h)(\phi_\lambda(z)) = \frac{1}{1-\conj{\lambda}z} \int_{\phi_{\lambda}(E)} \frac{(h\circ \phi_\lambda)(\zeta) \abs{\phi'_\lambda(\zeta)}(1-\lambda\conj{\zeta})}{1-\conj{\eta}\phi_\lambda (z)} dm(\zeta), \qquad z\in \D.
    \]
In a similar way, this calculation in conjunction with \thref{THM:Irred} implies that sets of rigidity are invariant under compositions with automorphisms $\phi_\lambda$, for any $\lambda\in \D$. Since any conformal self-map on $\D$ is a composition of a rotation and automorphism of the form $\phi_{\lambda}$, we conclude that sets of rigidity for $\B_0(w)$ are conformally invariant. We now turn to the conformal invariance of SA-sets, so let $E$ be an SA-set for $\B_0(w)$ and suppose for sake of obtaining a contradiction, that there exists an automorphism $\phi$ on $\D$ such that $\phi(E)$ is not an SA-set for $\B_0(w)$. According to \thref{THM:AbsThom}, there exists a subset $S \subseteq \phi(E)$ of positive Lebesgue measure, such that $S$ is a set of rigidity for $\B_0(w)$. But the conformal invariance of sets of rigidity in $\B_0(w)$ implies that $\phi^{-1}(S) \subseteq E$ is also a set of rigidity for $\B_0(w)$. But $E$ was assumed to be an SA-set for $\B_0(w)$ and the restriction principle for SA-sets in \thref{LEM:SA-RIG} therefore gives the desired contradiction. The proof is now complete.
\end{proof}

\subsection{The weak Khinchin-Ostrowkski property}
Our main emphasis here is to derive the relationship between SA-sets and sets of rigidity, to the weak Khinchin-Ostrowski property, in the broad context of holomorphic Banach spaces. First, we give the dual reformulation of the weak KO-property.

\begin{prop}\thlabel{PROP:DualKOw} Let $X$ be a holomorphic Banach space and let $E\subseteq \T$ be Lebesgue measurable set. Then $E$ satisfies the weak KO-property wrt $X$ if and only if the linear manifold 
\[
\mathcal{K}_E(X') := \left\{C(h)\in X': h \in L^1(E) \right\}
\]
is norm-dense in $C(L^1)$.
\end{prop}

\begin{proof}[Proof of \thref{PROP:DualKOw}]
Suppose that the linear manifold $\Ka_E(X')$ is dense in $C(L^1)$ and let $(Q_n)_n$ be analytic polynomials with $Q_n \to f$ in $X$ to some $f \in H^\infty$ and $Q_n \to 0$ in $L^\infty(E)$. Then for any $h \in L^1(E)$ with $C(h) \in X'$ we have 
\[
\lim_{r\to 1-} \int_{\T} f(\zeta) \conj{C(h)(r\zeta)} dm(\zeta) = \lim_n \lim_{r\to 1-} \int_{\T}Q_n\conj{C(h)} dm = \lim_n \int_E Q_n \conj{h} dm =0.
\]
Since such elements $C(h)$ are dense in $C(L^1)$, which is the pre-dual of $H^\infty$, we conclude that $f\equiv 0$ on $\D$. This proves that $E$ satisfies the weak KO-property wrt $X$. Conversely, suppose $\Ka_E(X')$ is not dense in $C(L^1)$. Recall that since $M_z ' : X' \to X'$, we have according to \thref{LEM:Mz*K} that
\[
\textbf{Clos}\left(\Ka_A(X') \right)_{C(L^1)} = \left\{ f= C(h): T_{\conj{\Theta}}(h) =0, \, \, h\in L^1(\T) \right\},
\]
for some inner function $\Theta$. A similar argument as in the proof of \thref{THM:Irred} shows that $h \equiv 0$ for any $h \in L^1(E)$ with $C(h) \in X'$ and therefore $E$ satisfies the SA-property wrt $X$ by Khrushchev's lemma. Consequently, $E$ cannot satisfy the weak KO-property wrt $X$.

\end{proof}
We complete the section by establishing the precise relationships between SA-sets, sets of rigidity and sets satisfying the weak KO-property, in the context of holomorphic Banach spaces.
\begin{thm}\thlabel{THM:KOprop} Let $X$ be a holomorphic Banach space and let $E\subseteq \T$ a Lebesgue measurable set. Then the following conditions are all equivalent.
\begin{enumerate}
    \item[(i)] $E$ does not satisfy the SA-property wrt $X$.
    \item[(ii)] $E$ almost contains a subset $A$ of positive Lebesgue measure, such that $A$ is a set of rigidity for $X$.
    \item[(iii)] $E$ satisfies the weak KO-property wrt $X$.
\end{enumerate}

\end{thm}

\begin{proof}
If $E$ almost contains no non-trivial subset $A$ which is a set of rigidity for $X$, then the quantity $\iota(X,E) =0$ in \thref{THM:AbsThom} and it follows from the same theorem that $E$ satisfies the SA-property wrt $X$, hence the weak KO-property of $E$ wrt $X$ cannot hold. This proves the chain of implications: $(iii) \implies (i) \implies (ii)$, hence it remains to prove that $(ii) \implies (iii)$. To this end, suppose that $\Po_A(X)$ is irreducible for some non-trivial subset $A \subseteq E$. According to the second statement of \thref{THM:Irred}, the linear manifold
\[
\Ka_A(X') = \left\{C(h) \in X': h \in L^1(A) \right\}
\]
is dense in $C(L^1)$. But \thref{PROP:DualKOw} asserts that this is equivalent to $E$ satisfying the weak KO-property wrt $X$, hence the proof is complete.
\end{proof}
\section{Existence of SA sets in weighted Bloch spaces}\label{SEC:4}
\subsection{Necessity of the square Dini condition}

We shall first give the short proof of the necessity of the condition \eqref{DiniDiv} for SA-set of positive Lebesgue measure in $\B_0(w)$. 

\begin{proof}[Proof of neccesity of \thref{THM:SAsets}]

Note that whenever $f\in \B_0(w)$, then for any arc $I$
\[
\int_{Q_I} \abs{f'(z)}^2 (1-|z|) dA(z) \lesssim \norm{f}^2_{\B(w)} |I| \int_0^{|I|} \frac{w^2(t)}{t} dt,
\]
where $Q_I = \{z\in \D: z/|z| \in I, 1-|z| \leq |I| \}$ denotes the Carleson square associated to $I$. Now if $\int_0^1 w^2(t)/t dt < \infty$ holds, then $\B_0(w)$ is contained in $VMOA$ (cf. \cite{cauchytransform} for definition) and thus also in the classical Hardy space $H^2$. One can now finish the proof by appealing to the classical Theorem of Khinchin and Ostrowski (for instance, see \cite{havinbook}), but for the sake of completeness we give a short and simple proof. Suppose $E\subset \T$ is a measurable set and $\{Q_n\}_n$ is a sequence of analytic polynomials with the properties that $Q_n \to f$ in $C(E)$ and $Q_n \to 0$ in $\B_0(w)$. By the previous argument, $Q_n \to 0$ in $H^2$, hence passing to a subsequence if necessary and by means of applying Egoroff's Theorem, we conclude that $f$ is a $H^2$-function which vanishes on $dm$-a.e $E$. Therefore, whenever condition \eqref{DiniDiv} fails, any measurable set of positive Lebesgue measure is a set of rigidity for $\B_0(w)$.
\end{proof}

\subsection{Main construction}
In this subsection, we shall under the assumption \eqref{DiniDiv} construct non-trivial $SA$-set for $\B_0(w)$. Our construction will be principally inspired from the work of S. Khrushchev (see Theorem 2.5 in \cite{khrushchev1978problem}). Let $0< \delta < 1$ and set $I_{\delta}=\{\zeta \in \T: |\zeta-1|\leq \delta/2\}$. Fix $0<\varepsilon<1$ and consider bounded real-valued functions $\psi_{\varepsilon,\delta}$ on $\T$ with the following properties:

\begin{enumerate}
    \item[(i)] $\int_{\T} \psi_{\varepsilon, \delta} dm = 0$. 
    \item[(ii)] $\psi_{\varepsilon, \delta} = \log \varepsilon$ on $\T \setminus I_{\delta}$.
\end{enumerate}


    
Define the corresponding outer function $F_{\varepsilon, \delta}$ on $\D$ with 
\[
\abs{F_{\varepsilon, \delta}(\zeta)} = \exp \psi_{\varepsilon, \delta}(\zeta) \qquad \zeta \in \T.
\]
For the sake of future references, we list the following obvious properties of $F_{\varepsilon, \delta}$ below.
\begin{enumerate}
    \item[(i)] $F_{\varepsilon, \delta}(0)=1$.
    \item[(ii)] $\abs{F_{\varepsilon, \delta}(\zeta)}= \varepsilon$ off the arc $I_\delta$.
    \item[(iii)] $\norm{F_{\varepsilon, \delta}}_{H^\infty} = N(\varepsilon, \delta)$, for some $N(\varepsilon, \delta) \to \infty$ if either $\varepsilon, \delta$ tend to $0$.
    \end{enumerate}
The first two properties are self-explanatory, while the second follows from weak-star compactness of the unit-ball in $H^\infty$. 

\begin{remark}\thlabel{REM:smoothcutoff} Note that one has tremendous freedom in determining the degree of smoothness of the functions $\psi_{\varepsilon, \delta}$. For instance, one can easily take them to be smooth on $\T$, which at its term gives rise to bounded analytic functions $F_{\varepsilon, \delta}$ on $\D$ with smooth extensions to the boundary $\T$. Indeed, this follows from fact that the Cauchy transform maps preserve smoothness on $\T$. Although, the precise smoothness properties of the functions $F_{\varepsilon, \delta}$ will be irrelevant in the proof of \thref{THM:SAsets}, it will useful when we shall phrase a general result in Section 4.3.
    
\end{remark}

For our purposes, we shall need the existence of inner functions, such that their hyperbolic derivative enjoy a certain prescribed radial decay. 

\begin{thm}[Theorem 6 in \cite{aleksandrov1999inner}] \thlabel{THM:wHypBloch} Let $w$ be a majorant satisfying 
\[
\int_0^1 \frac{w^2(t)}{t} dt = \infty.
\]
Then there exists an inner function $\theta$, which is not a finite Blaschke product, such that 
\[
\lim_{\abs{z} \to 1-} \frac{1}{w(1-|z|)} \frac{(1-|z|)\abs{\theta'(z)}}{1-\abs{\theta(z)}^2} =0. 
\]
\end{thm}
\noindent
See also \cite{smith1998inner} for a different construction of inner functions satisfying the above hypothesis. Inner functions of such type have the remarkable property that they induce composition operators which improve smoothness in the sense of membership in $\B_0(w)$. The following observation plays a crucial role for our further developments.

\begin{lemma} \thlabel{AANThm} Let $\eta >0$ and $w$ be a majorant satisfying 
\[
\int_0^1 \frac{w^2(t)}{t}dt = \infty.
\]
Then there exists a non-constant inner function $\theta_{\eta}$ with $\theta_{\eta}(0)=0$ such that the composition operator 
\[
C(\theta_\eta)f(z) = f(\theta_\eta(z)) \qquad z\in \D
\]
maps $H^\infty$ continuously into $\B_0(w)$ with operator norm $\norm{C(\theta_\eta)}\leq \eta$.
\end{lemma}
The proof is an immediate consequence of \thref{THM:wHypBloch} applied to the majorant $\eta w$ and exhibiting the corresponding inner function $\theta_\eta$ multiplied by $z$. See also Corollary 1 in \cite{aleksandrov1999inner} for a more general result. We are now ready to establish \thref{THM:SAsets}.

\begin{proof}[Proof of existence of SA-sets in \thref{THM:SAsets} ]
    
Set $G_{\varepsilon,\delta} := 1-F_{\varepsilon, \delta}$ as above and observe that $G_{\varepsilon,\delta}(0)=0$ while $\abs{G_{\varepsilon,\delta}-1} = \varepsilon$ off $I_{\delta}$. Fix a number $\eta>0$ to be specified in a moment and pick a non-constant inner function $\theta_\eta$ according to \thref{AANThm} such that the corresponding composition operator $C(\theta_\eta): H^\infty \to \B_0(w)$ with $\norm{C(\theta_\eta)}\leq \eta$. Consider the sets $E_{\varepsilon, \delta}(\eta)= \T \setminus \theta^{-1}_\eta(I_\delta)$, which on account of Proposition 9.1.18 in \cite{cauchytransform} and some customary modifications, may be taken to be Borel measurable. Recall that since $\theta_\eta(0)=0$, L\"owner's lemma ensures that $\abs{E_{\varepsilon, \delta}(\eta)}= 1 -\delta$. Now picking $\eta = \varepsilon/(1+N(\varepsilon,\delta) )$ and applying \thref{AANThm}, we get 
\[
\norm{C(\theta_\eta)G_{\varepsilon,\delta}}_{\B(w)} \leq \eta \norm{G_{\varepsilon,\delta}}_{H^\infty}\leq \varepsilon.
\]
Given a number $0<\gamma <1$, we may pick a sequence $\{\delta_j\}_j$ such that $\sum_j \delta_j \leq \gamma$. Now set $E := \cap_j E_{1/j, \delta_j}(\eta_j)$ where $\eta_j := N(1/j,\delta_j)/j$, and note that L\"owners lemma again implies that $E$ has measure no smaller than $1-\gamma$. Set $f_j = C(\theta_{\eta_j})G_{1/j, \delta_j}$ and observe that $\{f_j\}_j$ are bounded analytic functions in $\B_0(w)$ with $f_j \to 0$ in $\B_0(w)$ and $f_j \to 1$ uniformly on $E$. We shall now need the following lemma, whose proofs follows from a simple approximation argument involving dilations and is therefore omitted.

\begin{lemma}\thlabel{ApproxRed} Let $E \subset \T$ be a measurable set for which there exists a sequence $\{f_n\}_n$ of bounded analytic functions in $\B_0(w)$ such that 
\begin{enumerate}
    \item[(i)] $f_n$ converges to $0$ in $\B_0(w)$
    \item[(ii)] $f_n$ converges to $1$ uniformly on $E$.
\end{enumerate}
Then there exists analytic polynomials $\{Q_n\}_n$ such that
$Q_n$ converge to $0$ in $\B_0(w)$ and $Q_n$ converges to $1$ in $L^p(E)$ for every $p>0$. 
\end{lemma}


Continuing where we left off, we may according to \thref{ApproxRed} find a sequence of analytic polynomials $\{Q_j\}_j$ such that $Q_j$ converges to $0$ in $\B_0(w)$ and $Q_j$ converges to $1$ in $L^p(E)$, for any $p>0$. The proof now follows by passing to an appropriate subsequence and applying Egoroff's theorem. It follows that the tuple $(0,1)$ belongs to $\Po_E(\B_0(w))$, which completes the proof.
\end{proof}

\subsection{SA-sets for holomorphic Banach spaces} 

Here we summarize the crucial ingredients of our developments, which ensures that a holomorphic Banach space possesses non-trivial $SA$-sets.

\begin{thm}\thlabel{THM:SARegX} Let $X$ be a holomorphic Banach space with the additional property that it contains analytic functions with smooth extensions to $\T$. If there exists a sequence of inner functions $\{\theta_n\}_n$ such that the corresponding composition operators 
\[
C_{\theta_n}f(z) = f(\theta_n(z)), \qquad z \in \D
\]
satisfy 
\[
\lim_{n \to \infty} \norm{C_{\theta_n}f}_{X} = 0, \qquad f\in X \cap C^\infty(\T).
\]
Then there exists $SA$-sets wrt $X$ of any desirable Lebesgue measure $0<\delta<1$.
\end{thm}
The proof is a straightforward adaptation of the argument provided in \thref{THM:SAsets} in the context of weighted Bloch functions, where we essentially only need to make sure that building blocks $\psi_{\varepsilon, \delta}$ are chosen to be smooth on $\T$, see \thref{REM:smoothcutoff}. For instance, the theorem is applicable to holomorphic Banach spaces $X$ for which the following property holds:
\begin{equation}\label{EQ:Moncond}
\lim_{n \to \infty} \norm{f(z^n)}_X = 0, \qquad f\in X \cap C^\infty(\T).
\end{equation}
This fact was actually observed by S. Khrushchev, see Theorem 2.5 in \cite{khrushchev1978problem}. However,  condition \eqref{EQ:Moncond} fails for weighted Bloch spaces $\B_0(w)$, one may regard \thref{THM:SARegX} as a generalization of Khrushchev's result.


%
%
%
\section{SA-sets and removable sets}\label{SEC:5}

This section is devoted to establishing \thref{THM:CHAR1} and our description on removable sets for $W^1_a$ in \thref{THM:CHARREM}. 

\subsection{Function theoretical characterizations}

We immediately turn our attention to the proof of \thref{THM:CHAR1}.

\begin{proof}[Proof of \thref{THM:CHAR1}] 
Note that the equivalence between $(i)-(iii)$ actually follows from \thref{THM:KOprop}, hence it remains only to establish the equivalence between $(i)$ and $(iv)$. To this end, suppose $E$ is a set of positive Lebesgue measure, which does not satisfy the SA-property wrt $\B_0(w)$. Then by Khrushchev's lemma, there exists a function $h \in L^1(E)$ such that its Cauchy transform $C(h)$ is a non-zero element in $W^1_a(w)$. According to \thref{THM:BWchar}, the function 
\[
f(\zeta) := \lim_{r\to 1-} \Re C(h) (r\zeta ), \qquad \zeta \in \T
\]
defines an element in $\mathscr{B}_w(\T)$. Write $h= u +i v$, where $u, v \in L^1(E)$ are the real and imaginary parts of $h$, respectively. Observe that  
\[
\Re C(h)(z) = \int_{\T}P_z u dm - \int_{\T} Q_z v dm, \qquad z\in \D
\]
where $P_z$ denotes the Poisson kernel and $Q_z$ denotes the conjugate kernel. Using standard properties of boundary values, we conclude that 
\[
f=  u - H(v)
\] 
where $H$ denotes the Hilbert transform, which is enough to deduce that $(iii)$ holds. Conversely, suppose there exists a non-trivial pair of real-valued functions $u, v \in L^1(E)$ with the property that 
\[
u+ H(v) \in \mathscr{B}_w(\T) \setminus \{0\}
\]
Then $u+H(v)$ is non-constant, otherwise $0= H(u+H(v)) = H(u) -v$. According to \thref{THM:BWchar}, we have that that $C(u+H(v))$ belongs to $W^1_a(w)$, and is non-trivial since the sum $u + H(v)$ is assumed to be non-constant and real-valued. Consider the $L^1(\T)$-function $h:= u -iv$, which is also carried on the set $E$, and let $F:= C(h)$ denote its Cauchy transform. Note that by a straightforward expansion into Fourier series, we have
\[
\int_{\T}Q_z v dm = \int_{\T} P_z H(v) dm, \qquad z\in \D.
\]
This implies that 
\[
\Re F(z) = \Re C (u + H(v) )(z), \qquad z \in \D.
\]
This means that the real-part of $F= C(h)$ agrees with the real-part of a non-trivial element in $W^1_a(w)$, hence Khrushchev's lemma implies that $E$ is not an SA-set wrt $\B_0(w)$. Again, if $E$ is a carrier set of either $u$ or $v$, then $E$ is a also a carrier set of $h= u-iv$ and the claim on sets of rigidity follows from \thref{THM:Irred}.


\end{proof}

    

\begin{proof}[Proof of \thref{PROP:NecDR} and Beurling-Carleson entropy] Actually follows from the fact that the condition \eqref{EQ:Symcond} ensures that the indicator function $1_E$ belongs to the Besov space $B^1 \equiv \mathscr{B}_1$, hence an application of \thref{THM:CHAR1} implies that $E$ cannot be an SA-set for $\B_0$. Moreover, given a compact set $K$, we set 
\[
K_t := \left\{\zeta \in \T \setminus K: \textbf{dist}(\zeta, K) \leq t \right\},
\]
and observe that a simply argument involving geometry shows that 
\[
\abs{(K+t) \, \triangle \, K} \leq \abs{K_t}, \qquad 0<t<1.
\]
Recalling that $K$ has finite Beurling-Carleson entropy if and only if the condition
\[
\int_0^1 \abs{K_t} \, \frac{dt}{t} < \infty,
\]
holds (c.f \cite{ransforddbrdirichlet}, Exercise 9.4.1), we conclude that condition \eqref{EQ:Symcond} is satisfied for sets of finite Beurling-Carleson entropy.
\end{proof}

\subsection{Removable sets}




The proof of \thref{THM:CHARREM} is more involved and will therefore be carried out in several steps. We shall first need a simple lemma on division in the Bergman spaces of holomorphic functions in the unit disc.

\begin{lemma}\thlabel{LEM:DIVBERG}
There exists an absolute constant $c>0$, such that 
\[
\int_{\D} \abs{f(z)} \, \frac{dA(z)}{\abs{z}} \leq c \int_{1/2 \leq |z|< 1} \abs{zf(z)} \,  dA(z)
\]
for any holomorphic function $f$ on $\D$.
\end{lemma}
\begin{proof}
    The proof follows by transforming into polar coordinates and using the fact that the integral means
    \[
    [0,1) \ni r \mapsto \int_{\T} \abs{f(r\zeta) } dm(\zeta)
    \]
    are increasing in $0<r<1$, whenever $f$ is holomorphic on $\D$.
\end{proof}

\begin{proof}[Proof of \thref{THM:CHARREM}]

\proofpart{1}{$(ii)$ implies $(i)$:} Suppose that $K \subset \T$ is compact which is not removable for $W^1_a$. Then there exists a non-constant function $F$, holomorphic in $\C \setminus K$, such that 
\[
\int_{\C \setminus K} \abs{F'(z)} dA(z) < \infty.
\]
Now let $F_0:=F \lvert_{\D}$ and note that since $F_0 \in W^{1}_a(\D)$, we have according to \thref{PROP:W11B1} that the radial limit of $F_0$ belongs to $B^1(\T)$. Observe also that by a simple argument involving subharmonicity, we have that
\[
\abs{\lambda}^{3/2} \abs{F'(\lambda)} \lesssim \frac{\abs{\lambda}^{3/2}}{\text{dist}(\lambda, \T)^2} \int_{\abs{z}\leq \text{dist}(\lambda, \T)} \abs{F'(z)} dA(z) \lesssim \frac{1}{\sqrt{|\lambda|}} \int_{\C \setminus K} \abs{F'(z)} dA(z), \qquad \abs{\lambda}>2.
\]
This implies that $F'(\lambda)$ tends to zero at infinity and $F'$ is holomorphic there, hence so is $F$. We may therefore consider the holomorphic function $F_1$ on $\D$, defined by
\[
F_1(z) = \conj{F \lvert_{\C \setminus \D} (1/\conj{z}) }, \qquad z \in \D.
\]
A straightforward calculation involving a change of variable shows that
\[
\int_{\D} \abs{F_1 '(z)} dA(z) = \int_{\abs{z}>1} \abs{F'(z)} \frac{dA(z)}{|z|^2} \leq \int_{\abs{z}>1} \abs{F'(z)} dA(z) < \infty.
\]
This shows that $F_1$ belongs to $W^1_a(\D)$, and another application of \thref{PROP:W11B1} implies that the radial limit of $F_1$ also belongs to $B^1(\T)$. This implies that we can define the $B^1(\T)$-function 
\[
h(\zeta) := \lim_{r\to 1-} F_0(r\zeta) -\conj{F_1(r\zeta)} = \lim_{r\to 1-} f(r\zeta) -f(\zeta/r), \qquad \zeta \in \T.
\]
Since $F$ extends holomorphically off $K$, we must have that $h$ vanishes on $\T \setminus K$, which proves that $h$ is a $B^1(\T)$-function supported in $K$.

\proofpart{2}{$(i)$ implies $(ii)$:}
Let $h$ be a non-trivial function in $B^1(\T)$ supported in $K$, which we may assume is real-valued. Let $Q$ be a trigonometric polynomial with real-coefficients to be chosen later and note that the product $Q(\zeta) h(\zeta)$ is also an element in $B^1(\T)$, which according to \thref{PROP:W11B1} implies that its Cauchy transform 
\[
C(Q h)(z) := \int_{K} \frac{\zeta Q(\zeta) h(\zeta)}{\zeta -z} dm(\zeta) , \qquad z\in \D,
\]
is a non-trivial element in $W^{1}_a(\D)$. Observe that $C(Q h)$ extends holomorphically to $\C \setminus K$, hence in order to show that it belongs to $W^1_a(\C \setminus K)$, it remains to verify that 
\[
I=\int_{\abs{z}>1} \abs{C(Q h)'(z)} dA(z) < \infty.
\]
However, since $Qh$ is real-valued, a simple change of variable $z\mapsto 1/\conj{\lambda}$ and the identity 
\[
C(Q h)'(1/\conj{\lambda}) = \conj{\lambda^2 C(Qh)'(\lambda)} \qquad \lambda \in \D,
\]
readily implies that
\begin{equation}\label{EQ:Iinv}
I= \int_{\D}  \abs{C(Q h)'(1/\conj{\lambda})} \frac{dA(\lambda)}{\abs{\lambda}^4} = \int_{\D} \abs{ C(Q h)'(\lambda)} \frac{dA(\lambda)}{\abs{\lambda}^2}.
\end{equation}
The idea is now to choose the trigonometric polynomial $Q$ such that $C'(Qh)(0)=0$. We make the simple ansatz $Q(\zeta)= a\zeta + 1$ for some $a \in \mathbb{R}$ and seek to solve 
\[
0=C'(Qh)(0) = a\int_K h(\zeta) dm(\zeta) +  \int_K \conj{\zeta}h(\zeta) dm(\zeta).
\]
If $\int_{K}hdm=0$, then we set $a=0$, otherwise $a=- \int_K \conj{\zeta}hdm / \int_K h dm$. With that particular choice of $Q$, we have that $C(Q h)'(\lambda)/\lambda$ is holomorphic on $\D$, hence applying \thref{LEM:DIVBERG}, and recalling the expression in \eqref{EQ:Iinv}, we obtain 
\[
I = \int_{\D} \abs{ \frac{C(Qh)'(\lambda)}{\lambda}} \frac{dA(\lambda)}{\abs{\lambda}} \leq c \int_{1/2 \leq |z| <1} \abs{ C(Q h)'(\lambda)} dA(\lambda) < \infty.
\]
We therefore conclude that $C(Qh)$ belongs to $W^1_a(\C \setminus K)$. To complete the proof it remains to observe that  
\[
\partial_{\bar{z}} C( Qh)(z) = - c_1 zQ(z)h(z), \qquad z\in \C
\]
in the sense of distributions, where $\partial_{\bar{z}}= (\partial_x + i \partial_y)/2$, with $x,y$ being the real and imaginary parts of $z$, respectively, and $c_1 \neq 0$ is a numerical constant. By the assumption on $h$, this implies that $C(Qh)$ does not extend holomorphically across the carrier set of $h$, hence $K$ is not a removable set for $W^1_a$. 

\proofpart{3}{Equivalence between $(ii)$ and $(iii)$:}
If there exists a non-trivial element $g \in B^1(\T)$ which is supported on $K$, then evidently
\begin{equation}\label{EQ:gphi'}
\int_{\T} g(\zeta) \phi'(\zeta) dm(\zeta)=0, \qquad \forall \phi \in C^\infty_K(\T).
\end{equation}
It follows that $C_K^\infty(\T)$ cannot be dense in $\lambda_1(\T)$. Conversely, if $C_K^\infty(\T)$ is not dense in $\lambda_1(\T)$, then it follows by duality between $B^1(\T)$ and $\lambda_1(\T)$ and the Hahn-Banach separation Theorem, that there exists a non-trivial element $g\in B^1(\T)$ such that 
\[
\int_{\T} g(\zeta)\phi'(\zeta) dm(\zeta)=0, \qquad \forall \phi \in C^\infty_K(\T)
\]
holds. We conclude $g$ must vanish $dm$-a.e off $K$, hence the claim follows. Note that a similar argument also proves the equivalence to the statement that $C_K^\infty(\T)$ is weak-star dense in $\Lambda_1(\T)$.

\proofpart{4}{Negligible $\Lambda_1$-condenser capacity:} This is an immediate consequence of the density of $C_K^\infty(\T)$ is dense in the little-o Zygmund class $\lambda_1(\T)$ and the equivalence between $(i)$ and $(iii)$.


\end{proof}

Note that step $1$ in the proof of \thref{THM:CHARREM} actually gives a direct proof of \thref{THM:REMSA}, without invoking \thref{THM:CHAR1}. Indeed, since $F= C(1_K h)$ was shown to belong to $W^1_a(\D)$, Khrushchev's lemma implies that $K$ cannot satisfy the SA-property wrt $\B_0$.


\section{Sets of rigidity from entropy conditions} \label{SEC:6}

\subsection{A regular cut-off function}
Given a majorant $w$, we denote by $C_w(\T)$ the space of continuous functions $h$ on $\T$, satisfying 
\[
\norm{h}_{C_w} := \sup_{\substack{\zeta, \xi \in \T \\ \zeta \neq \xi}} \frac{\abs{h(\zeta)-h(\xi)}}{w\left( \abs{\zeta-\xi} \right)} < \infty,
\]
which consists of functions with modulus of continuity on $\T$ not exceeding $w$. The following lemma contains the main ingredient in our proof of \thref{THM:suffwG}, involving the construction of a certain cut-off function, which enjoys a certain regularity conditions.
\begin{lemma}\thlabel{wcutoff}
Let $w$ be a majorant and $K \subset \T$ be a set of finite $w$-entropy. Then there exists an outer function $f : \D \to \D$ such that $f$ extends analytically across $\T \setminus K$ and satisfies the estimate  
\begin{equation}\label{fKest}
    \abs{f(z)} \lesssim w \left( 1-|z| \right), \qquad z\in \partial \Gamma_K \cap \D,
\end{equation}
where $\Gamma_K := \cup_{\zeta \in K} \Gamma(\zeta)$ denotes a Stolz-Privalov domain associated to $K$. Furthermore, the function $F:\T \to \C$ defined by
\[
F(e^{it})=f(e^{it}) 1_{\T \setminus K}(e^{it}) 
\]
belongs to $C_w(\T)$, and there exists a constant $C>0$, such that
\begin{equation}\label{Fest}
\abs{F(e^{it})} \leq C  w(\textbf{dist}(e^{it},K))^N, \qquad 0\leq t \leq 2\pi,
\end{equation}
for any desirable $N>0$.
\end{lemma}

\begin{proof} We shall divide the proof in two steps.
\proofpart{1}{The construction of $f$:}
Let $I$ be a connected component of $\T \setminus K$ and consider the so-called Whitney decomposition of $I$ consisting of a collection of (half-open, say right end-point excluded) subarcs $\{I_n\}_{n \in \mathbb{Z}}$ of $I$ with the properties 
\[
\textbf{dist}(I_n, E) = \ell_n \asymp 2^{-|n|}m(I) \qquad n\in \mathbb{Z},
\]
where $\ell_n$ denotes the length of $I_n$. By means of re-labeling and for the sake of abbreviation, we shall in fact denote the joint collection of Whitney arcs of all connected components of $\T \setminus K$ by $\{I_n\}_n$. Let $\xi_n$ be the center of each $I_n$ and consider for each $n$ the analytic functions
\begin{equation}\label{DEF:h}
h_n(z) = \frac{\ell_n \log 1/w(\ell_n)}{1 + \ell_n - \conj{\xi_n}z}, \qquad z \in \D.
\end{equation}
We set $h= \sum_n h_n$ and define the function
\[
f(z) := \exp \left( -c h(z) \right), \qquad z \in \D
\]
where $c>0$ is a constant to be specified later. Note that for each $n$
\[
\Re h_n(z) =  \ell_n \log 1/w(\ell_n) \frac{1+\ell_n - \Re(\conj{z}\xi_n)}{\abs{1 + \ell_n - \conj{\xi_n}z}^2} >0 \qquad z \in \D,'
\]
which implies that $f:\D \to \D$ analytic. In fact, since the $\xi_n$'s only accumulate at $K$, the function $f$ actually extends analytically across $\T \setminus K$. Furthermore, it is straightforward to verify that $f$ is outer by means of checking the equality
\[
\log \abs{f(0)} = \int_ {\T} \log \abs{f(\zeta)} dm(\zeta).
\]
\proofpart{2}{Verifying the estimates \eqref{fKest} and \eqref{Fest}:}

Moving forward, one verifies by means of Taylor expansions that 
\begin{equation}\label{hnLocest}
\Re h_n(\zeta) \asymp \log 1/w(\ell_n), \qquad \zeta \in I_n
\end{equation}
with constants independent of $n$, hence there exists a universal constant $c'>0$ such that for any $\zeta \in I_n$, we have 
\[
\abs{f(\zeta)} \leq \exp (- c \Re h_n(\zeta) ) \leq w(\ell_n)^{cc'} \lesssim w\left( \textbf{dist}(\zeta,K)\right)^{cc'}.
\]
Now by means of choosing $c>0$ sufficiently large, the claim in \eqref{Fest} follows. To see the claim in \eqref{fKest}, we fix an arbitrary point $z = \abs{z}\zeta \in \partial \Gamma_K \cap \D$ and let $I_n$ denote the unique Whitney arc containing the radial projection $\zeta \in \T$ of $z$. Recall that by definition of the Stolz-Privalov domain $\Gamma_K$ we have
\begin{equation*}\label{h*est}
1-|z| \asymp \textbf{dist}(\zeta ,K) \asymp \ell_n \qquad z \in \partial \Gamma_K \cap \D, \, \, \zeta \in I_n.
\end{equation*}
According to \eqref{hnLocest}, we have
\[
\Re h_n(z) = \Re \frac{\ell_n \log 1/w(\ell_n)}{1 + \ell_n - \conj{\xi_n}z} \asymp   \log 1/w(\ell_n), \qquad z\in \partial \Gamma_K \cap \D, \, \, \zeta \in I_n.
\]
Using this, we obtain
\[
\abs{f(z)} \leq \exp (- c\Re h_n(\zeta) ) \leq w(\ell_n)^{cc'} \lesssim w\left( 1-|z| \right)^{cc'},  \qquad z\in \partial \Gamma_K \cap \D
\]
for some absolute constant $c'>0$. Again, by choosing $c>0$ even larger, the estimate in \eqref{fKest} now follows.

\proofpart{3}{The containment of $F$ in $C_w(\T)$:}
It now only remains to verify that $F= 1_{\T \setminus K} f$ belongs to $C_w(\T)$. To this end, we primarily note that $F$ is in fact continuous on $\T$ with $F\equiv 0$ on $K$ and smooth on $\T \setminus K$. Now observe that for any $\zeta \in K$ the estimate in \eqref{Fest} implies that 
\begin{equation}\label{westE}
\abs{F(\zeta)-F(\xi)} = \abs{F(\xi)} \leq C w\left( \textbf{dist}(\xi, K) \right) \leq C w\left( \abs{\zeta-\xi} \right).
\end{equation}
On the other hand, if $\zeta \in \T \setminus K$, then $\zeta$ is contained a unique arc $I_n$. Let $\widetilde{I}_n$ the union of $I_n$ and its two neighbouring Whitney arcs and observe that a simply calculation as before shows that
\[
\abs{h'(\zeta)} \lesssim 1 + \frac{\log 1/w(\ell_n)}{ \ell_n}, \qquad \zeta \in \widetilde{I}_n,
\]
with constant independent on $n$. This in conjunction with \eqref{Fest} gives the estimate
\[
\abs{F'(\zeta)} \lesssim \abs{h'(\zeta)} \abs{F(\zeta)} \lesssim \log 1/w(\ell_n) \frac{w(\ell_n)}{ \ell_n}, \qquad \zeta \in \widetilde{I}_n.
\]
Now since $w(t)/t \to \infty$, there exists a number $\delta_{\zeta}>0$ such that 
\[
\log 1/w(\ell_n) \frac{w(\ell_n)}{ \ell_n} \leq \frac{w(\delta_{\zeta})}{ \delta_{\zeta}}.
\]
With this at hand, it follows that whenever $\abs{\zeta- \xi}< \delta_{\zeta}$ one has
\begin{equation} \label{westnotE}
\abs{F(\zeta)-F(\xi)} \leq \max_{\eta \in \widetilde{I}_n} \abs{F'(\eta)} \abs{\zeta- \xi} \lesssim \log 1/w(\ell_n) \frac{w(\ell_n)}{ \ell_n} \abs{\zeta-\xi} \lesssim w(\abs{\zeta- \xi}).
\end{equation}
Combining \eqref{westE} and \eqref{westnotE}, we have derive the following statement: there exists a constant $C>0$ such that for any $\zeta \in \T$, we can find a small number $\delta_\zeta >0$ such that the estimate
\[
\abs{F(\zeta)-F(\xi)} \leq C w(\abs{\zeta-\xi}),
\]
holds whenever $\abs{\zeta- \xi} < \delta_{\zeta}$. Now covering $\T$ by the family of corresponding arcs $J_{\zeta} =\{\xi\in \T: \abs{\zeta-\xi}<\delta_{\zeta} \}$, a straightforward compactness argument yields that there exists constants $C, \delta>0$ such that
\begin{equation}
    \abs{F(\zeta)-F(\xi)} \leq C w(\abs{\zeta-\xi})
\end{equation}
holds whenever $\abs{\zeta-\xi} \leq \delta$. Now for points $\zeta,\xi \in  \T$ with $\abs{\zeta-\xi} > \delta$, we simply pick a finite partition $\{\zeta_j\}_{j=1}^{N_\delta}$ of the smallest arc joining $\zeta$ and $\xi$ with $\abs{\zeta_j- \zeta_{j+1}}\leq \delta$ and observe that the number of points $N_\delta$ does not exceed $1/\delta$. Applying the triangle inequality in a straightforward manner gives 
\[
\abs{F(\zeta)-F(\xi)} \leq \sum_k \abs{F(\zeta_k)-F(\zeta_{k+1})} \leq C \sum_{k}w(\delta) \lesssim \frac{C}{\delta}w(\abs{\zeta-\xi}).
\]
This completes the proof of the lemma.
\end{proof}

\subsection{An embedding of the Cauchy projection}

\begin{lemma}\thlabel{CwinW1}
Let $w$ be a majorant which satisfies the Dini-condition
\[
\int_0^1 \frac{w(t)}{t}dt < \infty.
\]
Then the Cauchy projection $C$ maps $C_w(\T)$ continuously into $W^{1}_a(\D)$.
\end{lemma}

\begin{proof}
Fix $\phi \in C_w(\T)$ and note that we can write
\[
C(\phi)'(z) = \int_{\T} \frac{\zeta \left(\phi(\zeta)-P(\phi)(z) \right)}{(\zeta-z)^2} dm(\zeta), \qquad z \in \D,
\]
where $P(h)(z)$ denotes the Poisson extension of $h$ to $\D$. With this at hand, we get
\[
(1-|z|^2)\abs{C(\phi)'(z)} \leq \norm{\phi}_{C_w} \int_{\T} \int_{\T} w(|\zeta-\xi|) P_z(\xi) P_z(\zeta) dm(\xi) dm(\zeta) , \qquad z \in \D.
\]
Now by sub-additivity of $w$ and symmetry of the Poisson kernel, we have
\[
\int_{\T} \int_{\T} w(|\zeta-\xi|) P_z(\xi) P_z(\zeta) dm(\xi) dm(\zeta) \lesssim  \int_{\T} w(|\zeta-z|) \frac{1-|z|}{\abs{\zeta-z}^2} dm(\zeta), \qquad z \in \D.
\]
Now consider the arc $I_z = \left\{\zeta\in \T: \abs{\zeta-z} < 1-|z| \right\}$ and note that since $w$ is monotonic, we have
\[
\int_{\T \setminus I_z } w(\abs{\zeta-z}) \frac{1-|z|}{\abs{\zeta-z}^2} dm(\zeta) \lesssim (1-|z|) \int_{1-|z|}^1 \frac{w(t)}{t^2} dt \lesssim w(1-|z|) \qquad z \in \D. 
\]
In the last step, we also utilized the observation that if $w(t)/t^{\alpha}$ is non-increasing for some number $0<\alpha<1$, then 
\[
\int_t^1 \frac{w(s)}{s^2} ds \lesssim \frac{w(t)}{t}, \qquad 0<t<1.
\]
On the other hand, it easily follows that
\[
\int_{I_z}  w(\abs{\zeta-z}) \frac{1-|z|}{\abs{\zeta-z}^2} dm(\zeta) \lesssim w(1-|z|), \qquad z \in \D.
\]
Combining, we arrive at
\[
\int_{\D} \abs{C(\phi)'(z)} dA(z) \lesssim \norm{\phi}_{C_w} \int_{\D} \frac{w(1-|z|)}{1-|z|} dA(z) \asymp \norm{\phi}_{C_w} \int_0^1 \frac{w(t)}{t} dt.
\]
\end{proof}

\subsection{Proof of \thref{THM:suffwG}}

\begin{proof}[Proof of \thref{THM:suffwG}]
\proofpart{1}{$K$ is a set of rigidity for $\B_0$:}
By means of substituting $w^A$ by $w$ and recalling that the $w$-entropy condition remains unchanged, we may assume that $w$ satisfies the Dini-condition 
\[
\int_0^1 \frac{w(t)}{t} dt < \infty.
\]
Let $K \subset \T$ be a set of finite $w$-entropy and recall that according to \thref{wcutoff}, we can find an outer function $f: \D \to \D$ such that $F := f 1_{\T \setminus K}$ belongs to $ C_w(\T)$. Note that by analyticity, we clearly have
\begin{equation}\label{conjtrick}
C(\conj{\zeta}f)(z) = \int_{\T} \frac{\conj{\zeta f(\zeta)}}{1-\conj{\zeta}z} dm(\zeta) =0, \qquad z\in \D.
\end{equation}
Rearranging this expression, we get 
\[
\int_{K} \frac{ \conj{f(\zeta) \zeta}}{1-\conj{\zeta}z} dm(\zeta) = - \int_{\T} \frac{ \conj{F(\zeta) \zeta}}{1-\conj{\zeta}z} dm(\zeta) = C(\conj{\zeta}F)(z) , \qquad z\in \D.
\]
However, since $\conj{F(\zeta)\zeta} \in C_w(\T)$, \thref{CwinW1} implies that its Cauchy transform  $C(\conj{\zeta}F)$ belongs to $W^1_a(\D)$, hence so does the left hand side in \eqref{conjtrick}. Consequently, the $L^1$-function $h(\zeta): = 1_{K}(\zeta)\conj{f(\zeta)\zeta}$ has carrier set $K$ and enjoys the property that $C(h)\in W^1_a$, which on account of \thref{THM:Irred} shows that $K$ is a set of rigidity for $\B_0$.
\proofpart{2}{Embedding into the Nevanlinna class:}
We primarily note that by the Dini assumption on $w$, we have 
\begin{equation}\label{EQ:wlog}
\sup_{0<t<1} w(t)\log (e/t) < \infty.
\end{equation}
To see this, note that 
\[
w(t) \log(e/t)  \leq \int_t^1 \frac{w(s)}{s} ds \leq \int_0^1 \frac{w(s)}{s} ds, \qquad 0<t<1.
\]
Let $\Gamma_K$ denote a Stolz-Privalov domain associated to $K$ and note that since $\Gamma_K$ is rectifiable, it follows from the F. and M. Riesz Theorem that $d\omega_z$ is mutually absolutely continuous wrt the arc-length measure on $\Gamma_K$, for each fixed $\lambda \in \Gamma_K$. Since $K$ has finite $w$-entropy, we may utilize \thref{wcutoff} and consider the holomorphic function $f_K: \D \to \D$ and observe that $f_K$ is then absolutely convergent wrt $d\omega_\lambda$ on $\partial \Gamma_K$. Fix an analytic polynomial $Q$ and note that by subharmonicity and the reproducing property of harmonic measure, we have
\[
\abs{f_K(\lambda)Q(\lambda)} \leq \int_{\partial \Gamma_K } \abs{f_K(z)Q(z)} d\omega_\lambda(z), \qquad \lambda \in \Gamma_K.
\]
To estimate the integral inside the disc $\D$, we use \eqref{EQ:wlog} in conjunction with \eqref{fKest} to get
\[
\int_{\partial \Gamma_K  \cap \D } \abs{f_K(z)Q(z)} d\omega_\lambda(z) \lesssim \norm{Q}_{\B} \int_{\partial \Gamma_K  \cap \D } \log \frac{1}{1-|z|} w(1-|z|) d\omega_\lambda(z) \lesssim \norm{Q}_{\B}.
\]
On the other hand, we also have
\[
\int_{\partial \Gamma_K \cap \T} \abs{f_K(z)Q(z)} d\omega_\lambda (z) = \int_{K} \abs{f_K(z)Q(z)} d\omega_\lambda (z) \lesssim \norm{f_K}_{H^\infty} \norm{Q}_{L^\infty(K)}.
\]
Since these bounds are uniform in $\lambda \in \Gamma_K$, we actually conclude that the multiplication with $f_K$ extends to a continuous linear operator from $\Po_K \B_0$ into the Hardy space of bounded analytic functions $H^\infty(\Gamma_K)$. This implies that for any element $(g,g^*) \in \Po_K \B_0$, $g$ has finite non-zero non-tangential limit $G$ on $\partial \Gamma_K$, which by means of possibly enlarging the aperture of each Stolz angle in $\Gamma_K$, must be an essentially bounded function there and agree with $g$ on $\partial \Gamma_K \cap \D$ and with $g^*$ $dm$-a.e on the subset $K := \partial \Gamma_K \cap \T$. In other words, $G\in L^\infty(\partial \Gamma_K)$ with $G=g^*$ $dm$-a.e on $K$ and $G=g$ on $\partial \Gamma_K \cap \D$. Now since, $d\omega_\lambda$ is absolutely continuous wrt arc-length measure on $\partial \Gamma_K$, we have by the reproducing formula that
\[
\abs{g(\lambda)} \leq \int_{\partial \Gamma_K} \abs{G(z)} d\omega_{\lambda}(z) \leq \norm{G}_{L^\infty(\Gamma_K)}, \qquad \lambda \in \Gamma_K.
\]
This gives the improved embedding $\Po_K \B_0 \hookrightarrow H^\infty(\Omega_K)$, and completes the proof of the Theorem. 
\end{proof}

We now argue for the case that the entropy condition alone cannot be further improved.

\begin{proof}[Proof of \thref{THM:SHARPENT}]
    Let $w$ be a majorant with the properties that 
    \[
    \int_0^1 \frac{w^2(t)}{t} dt = \infty,
    \]
    but such that 
    \[
    \int_0^1 \frac{w^{2+\delta}(t)}{t} dt < \infty,
    \]
    for some number $\delta>0$. Now let $K$ be set of positive Lebesgue measure on $\T$, which satisfies the SA-property wrt $\B_0(w)$, whose existence is guaranteed by \thref{THM:SAsets} in view of condition $(i)$ above. According to Khrushchev's lemma, we can pick analytic polynomials $(Q_n)_n$, which converge to $0$ in $\B_0(w)$ and weak-star in $L^\infty(K)$ to $1$. At the cost of slightly shrinking the Lebesgue measure of $K$, passing to an appropriate convex combination of the analytic polynomials $(Q_n)s$, applying Egoroff's theorem and appealing to the restriction principles in \thref{LEM:SA-RIG}, we may assume that $K$ is compact in $\T$, and that $Q_n$ converges uniformly to $1$ on $K$. Now by the trivial containment $\B_0(w) \subseteq \B_0$, we also have that $K$ satisfies the SA-property wrt $\B_0$. Recall that a straightforward application of the fundamental theorem of calculus in conjunction with condition $(ii)$ on the majorant $w$ above, implies that for any functions $f$ in $\B_0(w)$, the following radial growth condition holds:
\[
\abs{f(z)} \lesssim \norm{f}_{\B(w)} \int_{1-|z|}^1 \frac{w(t)}{t}dt \lesssim \norm{f}_{\B(w)}w(1-|z|)^{-(1+\delta)} \qquad z \in \D.
\]
However, this means that $Q_n$ also converges to $0$ in the growth space $\G_{w^{1+\delta}}$ above (see also the introduction for a formal definition), and uniformly to $1$ on $K$, hence $K$ also satisfies the SA-property wrt $\G_{w^{1+\delta}}$. According to Khrushchev's Theorem (see Theorem 3.1 in \cite{khrushchev1978problem}), this implies that $K$ cannot almost contain any set of finite $w$-entropy.
\end{proof}

\section{Approximation problems in de Branges-Rovnyak spaces}\label{SEC:7}
In this section, our principal intention is to utilize the developments on SA-sets for the weighted Bloch spaces to address certain questions on approximation in the context of de Branges-Rovnyak spaces.
\subsection{de Branges-Rovnyak spaces}
 Given an analytic function $b: \D \to \D$, we define the corresponding de Branges-Rovnyak space $\hil(b)$ as the Hilbert space of holomorphic functions in $\D$ supplied with the collection of reproducing kernels
\[
\kappa_b(z,\lambda) = \frac{1-\conj{b(\lambda)}b(z)}{1-\conj{\lambda}z} \qquad z,\lambda \in \D.
\]
For instance, the subcollection of de Branges-Rovnyak space $\hil(b)$ with symbols corresponding to inner functions $b$, parameterize the lattice of the non-trivial closed invariant subspace for the backward shift operator 
\[
M_z^*(f)(z) = \frac{f(z)-f(0)}{z}, \qquad z \in \D
\]
acting on functions $f$ belonging to the Hardy space $H^2$. These are commonly referred to as the model spaces, as the action of $M_z^*$ on the model spaces serve as functional models for a certain class of contractions acting on a separable Hilbert space. We refer the reader to \cite{cima2000backward} for detailed treatment on model spaces. A typical function theoretical feature of de Branges-Rovnyak spaces is their difficulty in identifying elements beyond their reproducing kernels. A distinguished case is when $b$, regarded as a vector in the unit-ball of $H^\infty$, is not an extreme-point. This is property is metrically characterized by the condition
\begin{equation}\label{Extcond}
\int_{\T} \log (1-|b(\zeta)|^2) dm(\zeta) > -\infty.
\end{equation}
In that case, one can show that the corresponding de Branges-Rovnyak space $\hil(b)$ contains all analytic polynomials and that they form a dense subset therein. For a general function $b$, A. Aleman and B. Malman established in 2017, that functions in $\hil(b)$ which extend continuously up to $\T$ form a dense subset of $\hil(b)$ (see \cite{dbrcont}), which extended a classical result of A.B. Aleksandrov, previously known in the context of model spaces \cite{aleksandrovinv}. More recently, questions on approximations with smooth functions on the boundary and their connection to the theory of subnormal operators were considered in \cite{DBRpapperAdem} and \cite{limani2024constructions}. Here, we shall consider the existence (or lack off) functions of enjoying weaker boundary smoothness properties in de Branges-Rovnyak spaces. To fix some notation, we let $b:\D \to \D$ be an element in the unit-ball and set $\Delta^2= 1-\abs{b}^2$ to denote the corresponding bounded weight on $\T$, whose carrier set is designated by
\[
E:=\left\{ \zeta \in \T: \Delta^2(\zeta)>0 \right\}.
\]
The following description of de Branges-Rovnyak spaces provides the right for studying such questions and has been proven prevalent in recent developments mention above. For instance, see \cite{jfabackshift}.

\begin{prop}\thlabel{PROP:HB} Let $b$ be an extreme-point in the unit-ball of $H^\infty$. A function $f\in H^2$ belongs to $\hil(b)$ if and only if there exists a unique element $f_+ \in L^2(\Delta^2 dm)$ such that 
\[
T_{\conj{b}}(f)(z) = \int_E \frac{f_+(\zeta)}{1-\conj{\zeta}z} \Delta^2(\zeta) dm(\zeta), \qquad z\in \D,
\] 
where $T_{\conj{b}}= C(\conj{b}\cdot)$ denotes the Toeplitz operator with symbol $\conj{b}$. The mapping $J:f \mapsto (f,f_+)$ from $\hil(b)$ into $H^2 \oplus L^2(E)$ defines an isometry which induces the following equivalent norm on $\hil(b)$:
\[
\norm{f}^2_{\hil(b)}= \norm{f}^2_{H^2} + \norm{f_+}^2_{L^2(E)}.
\]
Furthermore, an element $J(\hil(b))^\perp$ if and only if it is of the form $(bh,\Delta h)$ for some $h\in H^2$.
\end{prop}

\subsection{De Branges-Rovnyak spaces lacking smooth functions}

In \cite{limani2022abstract}, it was observed that it is possible to obtain the following slight extension of the aforementioned density theorem in $\hil(b)$.

\begin{thm}[See \cite{limani2022abstract}] \thlabel{THM:UAhB} Functions in $\hil(b)$ with uniformly convergent Taylor series on $\cD$ are dense in $\hil(b)$, for any $b$.
\end{thm}

We remark that a critical part of its proof hinges on results regarding the boundary behavior of Cauchy dual elements of space of uniformly convergent Taylor series due to S.A. Vinogradov, which at its turn relies in an essentially way on  Carleson's Theorem on pointwise convergence of Fourier series in $L^2$. It was asked in \cite{parise2019densite} whether it is possible to extend the density result to the space of analytic functions with absolutely convergent Taylor series. A counterexample to this fact was provided in the context of model spaces \cite{limani2022abstract}, which relies on the existence of cyclic inner functions in $\B_0$, see \cite{anderson1991inner} and some further recent work in \cite{limani2023mzinvariant}. It turns out that the existence of SA-sets in $\B_0(w)$ has some striking consequences to the context of extreme de Branges-Rovnyak spaces with an outer symbols $b$.

\begin{thm}\thlabel{THM:dBR} Let $w$ be a majorant satisfying \eqref{DiniDiv}. Then there exists an outer function $b_w: \D \to \D$ with the property that the corresponding de-Branges Rovnyak $\hil(b_w) \cap W^{1}_a(w) = \{0\}$. In particular, every non-trivial function in $\hil(b_w)$ satisfies
\begin{equation}\label{l^1_a(w)}
\sum_{n\geq 1} \abs{f_n} w(1/n) =+ \infty,
\end{equation}
where $(f_n)_n$ denotes the Taylor coefficients of $f$ (centered at the origin).
  
\end{thm}

To appreciate the ramification of \thref{THM:dBR}, one should contrast it to \thref{THM:UAhB}. For instance, if the majorant $w$ is constant, then there exists an extreme point $b$ in the unit-ball of $H^\infty$, such that the corresponding $\hil(b)$-space lacks non-trivial functions with absolutely convergent Taylor series. In fact, \thref{THM:dBR} implies the stronger statement that there exists an $\hil(b)$-space, which contains no non-trivial function $f$ satisfying 
\[
\sum_{n\geq 1} \frac{\abs{f_n} }{\sqrt{\log(n+1)\log \log (n+1) \cdot  \dots \cdot \log \log \dots \log (1+n)}} < \infty,
\]
for any number of desirable iterations of the $\log$. We now turn to the proof.
\begin{proof}[Proof of \thref{THM:dBR}]
Let $E$ be a set satisfying the SA-property wrt to $\B_0(w)$ and define the outer function $b_w=b:\D \to \D$ by 
\[
\abs{b(\zeta)} = \begin{cases} 1 & \zeta \in \T \setminus E \\
1/2 & \zeta \in E.
    
\end{cases}
\]
Then $b$ is clearly extreme, and according to \thref{PROP:HB}, a function $f \in H^2$ belongs to $\hil(b)$ if and only if there exists $f_+ \in L^2(E)$ such that 
\[
T_{\conj{b}}(f)(z) = \int_{\T} \frac{f_+(\zeta)}{1-\conj{\zeta}z} \Delta^2(\zeta) dm(\zeta), \qquad z\in \D.
\]
But since $F=T_{\conj{b}}(f)$ belongs to $\hil(b)$ whenever $f$ does (for instance, apply Fubini's Theorem and \thref{PROP:HB}), we get that 
\[
F(z)= \frac{3}{4} \int_{E} \frac{f_+(\zeta)}{1-\conj{\zeta}z}  dm(\zeta), \qquad z\in \D
\]
is an element in $\hil(b)$. But $E$ was assumed to satisfy the SA-property wrt $\B_0$, hence according to \thref{THM:SAsets} this is only possible if $f_+\equiv 0$ and thus $f\equiv 0$, since $T_{\conj{b}}$ is injective for $b$ outer. This proves that $W^1_a(w) \cap \hil(b)=\{0\}$. To prove the second statement, it suffices to show that every holomorphic function $f$ in $\D$ which satisfies 
\[
\sum_{n=0}^\infty \abs{f_n} w(1/n) < \infty
\]
must necessary belong to $W^1_a(w)$. By means of applying the triangle inequality, it suffices to verify that 
\[
n \int_{0}^1 r^n w(1-r) dr \lesssim w(1/n)
\]
for any majorant $w$. This task is accomplished by first using the monotonicity of $w$ to estimate
\[
n\int_{1-1/n}^1 r^n w(1-r) dr \leq n w(1/n) \int_0^1 r^n dr \lesssim w(1/n).
\]
For the second part, we utilize the majorant assumption of $w$, which implies the inequality $w(1-r) \leq (1-r)^\alpha n^\alpha w(1/n)$ for $1-r\geq 1/n$, where $0<\alpha <1$ is a fixed number. Using this in conjunction with well-known relationship between Euler Beta-functions and asymptotics of Gamma functions, we get 
\[
n\int_0^{1-1/n} r^n w(1-r) dr \lesssim w(1/n) n^{1+\gamma} \int_{0}^1 (1-r)^\gamma r^n dr \lesssim w(1/n).
\]
\end{proof}
\section{Applications to universality in the unit disc}\label{SEC:8}
In this section, we shall illustrate how our developments in the previous sections can be applicable to the theory of Universality. Our first objective is to deduce certain generic behavior of Taylor series of functions in regular spaces, using the Plessner-type theorem on polynomial approximation in \thref{THM:AbsThom} (see \thref{THM:Univ1} below). Secondly, we shall show that \thref{THM:SAsets} on the existence of $SA$-sets for certain weighted Bloch spaces allows us to answer questions that were raised in the recent work of \cite{beise2016generic}. This is summarized in \thref{THM:Univ2}. We also mentioned the recent work of S. Khrushchev, where his previous work on SA-sets was utilized in order to study universality of Taylor series, see \cite{khrushchev2020continuous}.

\subsection{General Theory of Universality}
To set up the framework, we give a brief and rough summary of the relevant notions in the theory of universality. For a more detailed treatment, we refer the reader to the excellent survey \cite{grosse1999universal}. Roughly speaking, the theory of universality deals with the study of limiting processes which give rise to the following philosophical aspect of universality: \emph{there exists an object which is "maximally divergent", such that via a countable process, can approximate a maximal class of objects}. To phrase this idea within an abstract mathematical framework, one considers a topological space $X$ of objects, a topological space $Y$ of elements which one would like to be approximated through a sequence of continuous mappings $\{T_n\}_n$ with $T_n : X \to Y$, $n=0,1,2,\dots$. An element $x \in X$ is said to be \emph{Universal} for a family (sequence) of operators $\{T_n\}_{n=0}^\infty$ if the set $\{T_n(x)\}_n$ is dense in $Y$, and we shall denote by $\mathcal{U} \subseteq X$ the set of universal elements and say that the family $\{T_n\}_n$ is universal if it admits a universal element. Several results in analysis has suggested the heuristic idea that whenever a universal element of family $\{T_n\}_n$ exists, then the set $\mathcal{U}$ of universal elements is in fact generic in the sense of Baire categories. We remind the reader that a set $X_0 \subset X$ is said to be generic in the sense of Baire categories, if $X_0$ is dense in $X$ and can be expressed as a countable intersection of open dense subsets of $X$. The fundamental Theorem of Universality goes as follows.

\begin{prop}[The Universality criterion]\thlabel{UnivCrit} Let $X,Y$ be complete separable metric spaces and $\{T_n\}_n$ be continuous mappings from $X$ into $Y$. Then the set $\mathcal{U}$ of universal elements for $\{T_n\}_n$ is generic in $X$ in the sense of Baire category, if and only if to every pair of $(x,y)\in X \times Y$ there exists a sequence $\{x_k\}_k$ in $X$ and a subsequence $\{T_{n_k}\}_k$ such that 
\[
(x_k,T_{n_k}(x_k)) \to (x,y)
\]
in $X \times Y$.

\end{prop}

The proof of the above statement is actually not difficult to prove and we refer the reader to \cite{grosse1999universal} and references therein, for further details on these matters.


\subsection{Mensov Universality of Taylor series}
Our principal application of \thref{UnivCrit} will be for a family $\{T_n\}_n$ which correspond to the Taylor polynomials 
\[
T_N(f)(z) = \sum_{n=0}^N f_n z^n \qquad N=0,1,2,\dots, \qquad z \in \D
\]
which we regard as linear operators acting on a holomorphic Banach space $X$ into $L^\infty(E)$ for some Lebesgue measurable set $E\subset \T$. To this end, we shall need a certain notion of universality for Taylor polynomials due to Menshov.
\begin{definition}\label{DEF:MenUniv} A Lebesgue measurable subset $E \subseteq \T$ is said to induce \emph{Menshov universality} in $X$, if there exists an element $f \in X$ satisfying the following property: every Lebesgue measurable function $g$ on $E$ can be approximated pointwise $dm$-a.e on $E$, by a subsequence $\{T_{n_k}(f)\}_k$ of Taylor polynomials of $f$. In that case, we shall say that $(f,E)$ is a universal Menshov-pair from $X$.

\end{definition}
Our main result in this section gives a fairly sharp dichotomy for Lebesgue measurable subset of $\T$, wrt Menshov universality in the context of holomorphic Banach spaces, which bears the probabilistic flavor of a zero-one-law. 

\begin{thm}\thlabel{THM:Univ1} Let $X$ be a holomorphic Banach space and $E \subset \T$ be a Lebesgue measurable set. Then there exists a  partition $A, S$ of $E$ (unique modulo sets of Lebesgue measure zero), which gives rise to the following distinct phenomenons.
\begin{enumerate}
    \item[(i)] There exists a generic subset $\mathcal{U}_X \subset X$ in the sense of Baire categories, such that every $f \in \mathcal{U}_X$ gives rise to a universal Menshov-pair $(f,S)$ from $X$. 
    
    \item[(ii)] Whenever $(f,f^*) \in \Po_A(X)$ and there exists a subsequence $\{T_{n_k}f\}_k$ which converges pointwise $dm$-a.e on $A$, then they must converge to $f^*$.
 
\end{enumerate}

\end{thm}
\begin{proof}
\proofpart{1}{Statement $(i)$:}
According to \thref{THM:AbsThom}, there exists a unique (modulo sets of Lebesgue measure zero) partition of Lebesgue measurable subsets $A, S$ of $E$, such that $\Po_A(X)$ is irreducible and $S$ satisfies the SA-property wrt $X$. We may assume that both $A, S$ are non-trivial, that is, they have positive Lebesgue measure, otherwise one can set one of them to be the empty-set, for which the verification of either $(i)$ or $(ii)$ becomes redundant. To prove $(i)$, we shall adapt the argument in \cite{beise2016generic}. Note that since any measurable function $f$ is the pointwise $dm$-a.e limit of its truncation by height $f_N := f 1_{\abs{f}\leq N}$, it suffices to approximate any essentially bounded functions on $S$. To this end, we shall regard the sequence of Taylor polynomials $\{T_n\}_n$ as continuous linear operators from $X$ into balls of radius $M>0$ in $L^\infty(S)$, which equipped with the weak-star topology are metrizable (for instance, see \cite{rudin1973functional}). We denote the ball by $B_M(L^\infty(A))$ and recall that the universality criterion in \thref{THM:Univ1} assures that it suffices to check that for any pair $(f,g) \in X \oplus B_M(L^\infty(A))$ there exists analytic polynomials $\{Q_k\}$ and a subsequence $\{T_{n_k}\}_k$ such that
\[
(Q_k, T_{n_k}(Q_k) ) \to (f,g) 
\]
in $X \oplus B_M(L^\infty(A))$. However, since $T_{n} (Q_k) = Q_k$ for any $n$ sufficiently large, the universality criterion actually follows from the assumption that $S$ satisfies the SA-property wrt $X$. By means of repeat the same argument for any ball $B_M(L^\infty(A))$, the claim follows.

\proofpart{2}{Statement $(ii)$:}
Let $(f,f^*)\in \Po_A(X)$ and assume that $\{T_{n_k}f\}_k$ is a subsequence of Taylor polynomials of $f$, which converges pointwise $dm$-a.e to a measurable function $g$. Pick a sequence of analytic polynomials $(Q_j)_j$ such that $(Q_j, Q_j)$ converges to $(f, f^*)$ in $X \oplus L^\infty(A)$. Fix an $\varepsilon >0$ and write
\[
m(\{ A: \abs{g-Q_j} > \varepsilon \}) \leq m(\{ A: \abs{g-T_{n_k}(f)} > \varepsilon/2 \}) + \frac{2}{\varepsilon} \sum_{m=0}^{n_k} \abs{f_m -(Q_j)_m}.
\]
Using the assumption that $Q_j \to f$ in $X \hookrightarrow \text{Hol}(\D)$, it follows by standard arguments involving uniform convergence on compacts, that the Taylor coefficients $(Q_j)_n$ of $Q_j$ converge to the Taylor coefficients $f_n$ of $f$. Sending $j$ to infinity, we obtain
\[
\limsup_{j \to \infty} m(\{ A: \abs{g-Q_j} > \varepsilon \}) \leq m(\{ A: \abs{g-T_{n_k}(f)} > \varepsilon/2 \}), \qquad \forall k, \varepsilon >0.
\]
Now since $T_{n_k}(f)$ converges pointwise $dm$-a.e to $g$ on $A$, they also converge in Lebesgue measure on $A$, hence we can send $k \to \infty$ and conclude that $Q_j \to g$ in Lebesgue measure on $A$. By means of passing to a subsequence and using the assumption that $Q_j \to f^*$ weak-star in $L^\infty(A)$, we conclude that $g=f^*$ on $A$, which proves $(ii)$. The proof is now complete.
\end{proof}

Recall that examples of Holomorphic Banach spaces $X$ with SA-sets of positive Lebesgue measure are provided by \thref{THM:SARegX}, hence the above theorem has particularly interesting applications for such spaces. In particular, it applies to weighted Bloch spaces $\B_0(w)$ with majorants $w$ satisfying condition \eqref{DiniDiv}. It was asked in \cite{beise2016generic} (see Remark 2.3 therein) whether there exists Lebesgue measurable subsets $E \subset \T$ which induce Menshov universality in $\B_0$. An application of \thref{THM:Univ1} in conjunction with \thref{THM:SAsets}, answers the question affirmatively in a strong sense.

\begin{cor}\thlabel{THM:Univ2} Let $w$ be a majorant satisfying 
\[
\int_{0}^1 \frac{w^2(t)}{t} dt = +\infty.
\]
Then for any number $0<\delta <1$, there exists a set $E_{\delta}$ of Lebesgue measure $\delta$ and a corresponding subset $\mathcal{U}_{w,\delta}$ of $\B_0(w)$, generic in the sense of Baire categories, such that every element $f\in \mathcal{U}_{w,\delta}$ gives rise to a universal Menshov pair $(f, E_{\delta})$ from $\B_0(w)$.
\end{cor} 
It was further asked in \cite{beise2016generic}, which conditions on such sets $E$ in give rise to the Mensov universality in $\B_0$. Our \thref{THM:CHAR1} gives a function theoretical description of such sets, while \thref{THM:CHARREM} and \thref{COR:Hwinv} give some necessary metric and/or geometric conditions.

\bibliographystyle{siam}
\bibliography{mybib}

\end{document}